\documentclass[12pt]{amsart}
\usepackage{amsmath}
\usepackage{amsxtra}
\usepackage{enumitem} 
\usepackage{amscd}
\usepackage{amsthm}
\usepackage{hyperref}
\usepackage[T1]{fontenc}
\usepackage{amsfonts}
\usepackage{amssymb}
\usepackage{eucal}

\usepackage{color} 
\newtheorem{thm}[subsection]{Theorem}
\newtheorem{cor}[subsection]{Corollary}
\newtheorem{lem}[subsection]{Lemma}
\newtheorem{prop}[subsection]{Proposition}

\theoremstyle{definition}

\newtheorem{conj}[subsection]{Conjecture}
\newtheorem{defn}[subsection]{Definition}

\newtheorem{rem}[subsection]{Remark}

\usepackage{yhmath}
\DeclareSymbolFont{largesymbols}{OMX}{yhex}{m}{n}
\DeclareMathAccent{\widetilde}{\mathord}{largesymbols}{"65}

\newcommand{\thmref}[1]{Theorem~\ref{#1}}
\newcommand{\secref}[1]{\S~\ref{#1}}
\newcommand{\lemref}[1]{Lemma~\ref{#1}}

\newcommand{\propref}[1]{Proposition~\ref{#1}}

\newcommand{\conjref}[1]{Conjecture~\ref{#1}}
\newcommand{\remref}[1]{Remark~\ref{#1}}

\newcommand{\nc}{\newcommand}
\nc{\renc}{\renewcommand}
\nc{\ssec}{\subsection}
\nc{\sssec}{\subsubsection} 
\nc\ol{\overline}
\nc\wt{\widetilde}
\nc\wh{\widehat}
\nc\tboxtimes{\wt{\boxtimes}}

\renc{\d}{{\delta}}
\nc{\Aa}{{\mathbb{A}}}
\nc{\Bb}{{\mathbb{B}}}
 \nc{\Gg}{{\mathbb{G}}}  
\nc{\Hh}{{\mathbb{H}}}
 \nc{\Nn}{{\mathbb{N}}}
\nc{\Pp}{{\mathbb{P}}}
\nc{\Rr}{{\mathbb{R}}}
\nc{\BV}{{\mathbb{V}}}
\nc{\BW}{{\mathbb{W}}}
\nc{\Zz}{{\mathbb{Z}}}
\nc{\Qq}{{\mathbb{Q}}}
\nc{\Ss}{{\mathbb{S}}}
\nc{\Cc}{{\mathbb{C}}}
\nc{\Ff}{{\mathbb{F}}}
 \nc{\EL}{{L_\infty}}

\nc{\CA}{{\mathcal{A}}}
\nc{\CB}{{\mathcal{B}}}

\nc{\CE}{{\mathcal{E}}}
\nc{\CF}{{\mathcal{F}}}

\nc{\Las}{\mathsf{Las}}
\nc{\CG}{{\mathcal{G}}}

\nc{\CL}{{\mathcal{L}}}
\nc{\CC}{{\mathcal{C}}}
\nc{\CM}{{\mathcal{M}}}

\nc{\CN}{{\mathcal{N}}}
\nc{\Oog}{{\mathbb{O}}}
\nc{\Oo}{{\mathcal{O}}}
\nc{\CP}{{\mathcal{P}}}
\nc{\CQ}{{\mathcal{Q}}}
\nc{\CR}{{\mathcal{R}}}
\nc{\CS}{{\mathcal{S}}}
\nc{\CT}{{\mathcal{T}}}
\nc{\CU}{{\mathcal{P}}}

\nc{\CV}{{\mathcal{V}}}
 
\nc{\CW}{{\mathcal{W}}}
\nc{\CZ}{{\mathcal{Z}}}

\nc{\cM}{{\check{\mathcal M}}{}}
\nc{\csM}{{\check{\mathcal A}}{}}
\nc{\oM}{{\overset{\circ}{\mathcal M}}{}}
\nc{\obM}{{\overset{\circ}{\mathbf M}}{}}
\nc{\oCA}{{\overset{\circ}{\mathcal A}}{}}
\nc{\obA}{{\overset{\circ}{\mathbf A}}{}}
\nc{\ooM}{{\overset{\circ}{M}}{}}
\nc{\osM}{{\overset{\circ}{\mathsf M}}{}}
\nc{\vM}{{\overset{\bullet}{\mathcal M}}{}}
\nc{\nM}{{\underset{\bullet}{\mathcal M}}{}}
\nc{\oD}{{\overset{\circ}{\mathcal D}}{}}
\nc{\obD}{{\overset{\circ}{\mathbf D}}{}}
\nc{\oA}{{\overset{\circ}{\mathbb A}}{}}
\nc{\op}{{\overset{\bullet}{\mathbf p}}{}}
\nc{\cp}{{\overset{\circ}{\mathbf p}}{}}
\nc{\oU}{{\overset{\bullet}{\mathcal U}}{}}
\nc{\oZ}{{\overset{\circ}{\mathcal Z}}{}}
\nc{\ofZ}{{\overset{\circ}{\mathfrak Z}}{}}
\nc{\oF}{{\overset{\circ}{\fF}}}

\nc{\fa}{{\mathfrak{a}}}
\nc{\fb}{{\mathfrak{b}}}
\nc{\fg}{{\mathfrak{g}}}
\nc{\fgt}{{\fg}_!}
\nc{\fgl}{{\mathfrak{gl}}}
\nc{\fh}{{\mathfrak{h}}}
\nc{\fj}{{\mathfrak{j}}}
\nc{\fm}{{\mathfrak{m}}}
\nc{\ft}{{\mathfrak{t}}}
\nc{\fn}{{\mathfrak{n}}}
\nc{\fu}{{\mathfrak{u}}}
\nc{\fp}{{\mathfrak{p}}}
\nc{\fr}{{\mathfrak{r}}}
\nc{\fs}{{\mathfrak{s}}}
\nc{\fsl}{{\mathfrak{sl}}}
\nc{\hsl}{{\widehat{\mathfrak{sl}}}}
\nc{\hgl}{{\widehat{\mathfrak{gl}}}}
\nc{\hg}{{\widehat{\mathfrak{g}}}}
\nc{\chg}{{\widehat{\mathfrak{g}}}{}^\vee}
\nc{\hn}{{\widehat{\mathfrak{n}}}}
\nc{\chn}{{\widehat{\mathfrak{n}}}{}^\vee}

\nc{\fA}{{\mathfrak{A}}}
\nc{\fB}{{\mathfrak{B}}}
\nc{\fD}{{\mathfrak{D}}}
\nc{\fE}{{\mathfrak{E}}}
\nc{\fF}{{\mathfrak{F}}}
\nc{\fG}{{\mathfrak{G}}}
\nc{\fK}{{\mathfrak{K}}}
\nc{\fL}{{\mathfrak{L}}}
\nc{\fM}{{\mathfrak{M}}}
\nc{\fN}{{\mathfrak{N}}}
\nc{\fP}{{\mathfrak{P}}}
\nc{\fU}{{\mathfrak{U}}}
\nc{\fV}{{\mathfrak{V}}}
\nc{\fZ}{{\mathfrak{Z}}}

\nc{\bb}{{\mathbf{b}}}
\nc{\bc}{{\mathbf{c}}}
\nc{\bd}{\partial}
\nc{\be}{{\mathbf{e}}}
\nc{\bj}{{\mathbf{j}}}
\nc{\bn}{{\mathbf{n}}}
\nc{\bp}{{\mathbf{p}}}
\nc{\bq}{{\mathbf{q}}}
\nc{\bF}{{\mathbf{F}}}
\nc{\bu}{{\mathbf{u}}}
\nc{\bv}{{\mathbf{v}}}
\nc{\bx}{{\mathbf{x}}}
\nc{\bs}{{\mathbf{s}}}
\nc{\by}{{\bar{y}}}
\nc{\bw}{{\mathbf{w}}}
\nc{\bA}{{\mathbf{A}}}
\nc{\bK}{{\mathbf{K}}}
\nc{\bI}{{\mathbf{I}}}
\nc{\bB}{{\mathbf{B}}}
\nc{\bG}{{\mathbf{G}}}
\nc{\bC}{{\mathbf{C}}}
\nc{\bD}{{\mathbf{D}}}
\nc{\bP}{{\mathbf{P}}}
\nc{\bH}{{\mathbf{H}}}
\nc{\bM}{{\mathbf{M}}}
\nc{\bN}{{\mathbf{N}}}
\nc{\bV}{{\mathbf{V}}}
\nc{\bU}{{\mathbf{U}}}
\nc{\bL}{{\mathbf{L}}}
\nc{\bT}{{\mathbf{T}}}
\nc{\bW}{{\mathbf{W}}}
\nc{\bX}{{\mathbf{X}}}
\nc{\bY}{{\mathbf{Y}}}
\nc{\bZ}{{\mathbf{Z}}}
\nc{\bS}{{\mathbf{S}}}
\nc{\bSi}{{\bar{\Sigma}}}
\nc{\sA}{{\mathsf{A}}}
\nc{\sB}{{\mathsf{B}}}
\nc{\sC}{{\mathsf{C}}}
\nc{\sD}{{\mathsf{D}}}
\nc{\sF}{{\mathsf{F}}}
\nc{\sG}{{\mathsf{G}}}
\nc{\sK}{{\mathsf{K}}}
\nc{\sM}{{\mathsf{M}}}
\nc{\sO}{{\mathsf{O}}}
\nc{\sQ}{{\mathsf{Q}}}
\nc{\sP}{{\mathsf{P}}}
\nc{\sZ}{{\mathsf{Z}}}
\nc{\sfp}{{\mathsf{p}}}
\nc{\sr}{{\mathsf{r}}}
\nc{\sg}{{\mathsf{g}}}
\nc{\sff}{{\mathsf{f}}}
\nc{\sfb}{{\mathsf{b}}}
\nc{\sfc}{{\mathsf{c}}}
\nc{\sd}{{\ltimes}} 

\nc{\tH}{{\widetilde{H}}}
\nc{\tA}{{\widetilde{\mathbf{A}}}}
\nc{\tB}{{\widetilde{\mathcal{B}}}}
\nc{\tg}{{\widetilde{\mathfrak{g}}}}
\nc{\tG}{{\widetilde{G}}}
\def\T{{\mathsf{T}}}
\def\Tt{\T}
\nc{\TM}{{\widetilde{\mathbb{M}}}{}}
\nc{\tO}{{\widetilde{\mathsf{O}}}{}}
\nc{\tU}{\widetilde{U}}
\nc{\TZ}{{\tilde{Z}}}
\nc{\tx}{{\tilde{x}}}
\nc{\tq}{{\tilde{q}}}

\nc{\tfP}{{\widetilde{\mathfrak{P}}}{}}
\nc{\tz}{{\tilde{\zeta}}}
\nc{\tmu}{{\tilde{\mu}}}

 \def\e{\epsilon}

  \nc{\vol}{{\mathop{\operatorname{\rm vol\,}}}}
  
  \nc{\gal}{{\mathop{\operatorname{\rm Gal\,}}}}
  \nc{\cl}{{\mathop{\operatorname{\rm cl}}}}
  \nc{\disc}{{\mathop{\operatorname{\rm disc}}}}
  \nc{\Sym}{{\mathop{\operatorname{\rm Sym}}}}
   \nc{\Aut}{{\mathop{\operatorname{\rm Aut}}}}
 \nc{\Spec}{{\mathop{\operatorname{\rm Spec}}}}
  \nc{\spec}{{\mathop{\operatorname{\rm Spec}}}}
\nc{\Ker}{{\mathop{\operatorname{\rm Ker}}}}
 \nc{\dom}{{\mathop{\operatorname{\rm dom}}}}
\nc{\End}{{\mathop{\operatorname{\rm End}}}}
 \nc{\Hom}{\operatorname{\Hom}}
 \nc{\GL}{{\mathop{\operatorname{\rm GL}}}}
 \nc{\Id}{{\mathop{\operatorname{\rm Id}}}}
 \nc{\rk}{{\mathop{\operatorname{\rm rk}}}}
 \nc{\length}{{\mathop{\operatorname{\rm length}}}}
\nc{\supp}{{\mathop{\operatorname{\rm supp} \, }}}
\nc{\val}{{\rm val}}
\nc{\res}{{\mathop{\operatorname{\rm res}}}}

\def\Ind#1#2#3{{#1} {\downarrow}_{#3} {#2} }

\def\meet{\cap}
\def\union{\cup}

\def\si{\sigma}

\def\Sum{\Sigma}

\def\m{\smallsetminus}

\nc{\seq}[1]{\stackrel{#1}{\sim}}

\def\inv{^{-1}}

\def\beq#1{\begin{equation} \label{ #1}}
\def\eeq{\end{equation}}

\def\prf{\begin{proof}}
\def\pv{\end{proof} }
 \def\eprf{\end{proof} }

\def\acl{\mathop{\rm acl}\nolimits}

 \def\lbl#1{ { \color{green}{[#1]}}  \label{#1}  }
 \def\lbl#1{     \label{#1}  }
\def\a{\alpha}

 \renc{\b}{{\beta}}

 \def\hZ{{\widehat{\Zz}}}

\def\Ind#1#2{#1\setbox0=\hbox{$#1x$}\kern\wd0\hbox to 0pt{\hss$#1\mid$\hss}
\lower.9\ht0\hbox to 0pt{\hss$#1\smile$\hss}\kern\wd0}

 \def\Lam{\Lambda}

 \def\om{\omega}

\usepackage{color}

 \def\club{\clubsuit}    \def\spade{\spadesuit}

\usepackage{color}
\usepackage{hyperref}

 \def\uA{\underline{A}} 

\usepackage{url}
\def\lam{\lambda}  
 
  \def\PFp{\mathsf{PF}_+}  
  \def\SP{\mathsf{SP}}
 \def\Pfp{\PFp}
 \def\Pfpw{\mathsf{PF}_{\widehat{+}}}
  
   \def\PPp{\mathsf{PP}_+}
   \def\Ppp{\PPp}

\title{Ax's theorem with an additive character}
\author{Ehud Hrushovski }

\begin{document}

\begin{abstract}   Motivated by  Emmanuel Kowalski's  exponential sums
 over definable sets in finite fields, we
 generalize Ax's theorem on pseudo-finite fields  to a continuous-logic setting allowing for 
  an additive character.    
   The role played by Weil's Riemann hypothesis for curves over finite fields is taken 
by the `Weil bound' on exponential sums.   
     Subsequent model-theoretic developments, including simplicity and the Chatzidakis-Van den Dries-Macintyre definable measures,  also generalize.

    Analytically, we have the  following consequence:   consider 
      the algebra of functions $\Ff_p^n \to \Cc$ obtained from the  additive characters  and the characteristic functions of 
      subvarieties   by pre-  or post-composing with polynomials,  
     applying min and sup  operators to the real part, and  averaging   over subvarieties.  
 Then any element of this class can be  approximated, uniformly in the variables and in the prime $p$, by a polynomial expression in $\Psi_p(\xi)$  at certain
 algebraic functions $\xi$ of the variables, where $\Psi(n \mod p) = exp(2 \pi i n/p)$ is the standard additive character.
 
  \end{abstract}
\maketitle

\begin{section}{Introduction}
\label{intro}
The first-order theory of the class of finite fields was determined in a fundamental paper of James Ax (\cite{ax}).
He understood that the basic definable sets must be taken to be, not just the points of algebraic varieties, but
also their finite projections; in line with the \'etale theory that Grothendieck and his school were developing in these years.
Once this is granted, Ax showed that Weil's Riemann hypothesis for curves leads to a quantifier-elimination result.
Finally, having reduced to quantifier-free formulas,  the set of sentences true  in every (or in almost every) finite field is determined by the Chebotarev density theorem.  

Ax's student Kieffe studied zeta functions in this connection.  A   definable measure, corresponding to 
 the leading coefficient of the asymptotic expression for the cardinality of a definable set,  was shown to exist and 
 have good properties  by Chatzidakis, van den Dries and Macintyre in \cite{cdm}.

 Now Weil wrote in 1948 about a ``connection between various types of exponential sums, occurring in number-theory, and the so-called Riemann hypothesis in function-fields" \cite{weil}.    An exponential sum is an expression such as 
$\sum_{x \in X} \chi(f(x))$, where $X$ is (for our purposes) a subset of $F^n$, $F$ a finite field, $f$ a polynomial,
and  $\chi$  an additive character,   i.e. a homomorphism $(\Ff_q,+) \to (\Tt,\cdot)$, where $\Tt=\{z: |z|=1\}$ is the unit circle in $\Cc$.   
 On $\Ff_p$, the group of characters is generated by the canonical one,  $n \mod p \mapsto e^{2 \pi i n/p}$.  
From his work on the Riemann hypothesis for curves, Weil was able to derive
a on exponential sums.    When $X=C(F)$ with $C$   an absolutely irreducible  over $F=\Ff_q$, $f$ is not
constant on $C$ 
 and $\chi \neq 1$, he showed
 \[ \frac{1}{q} \sum_{c \in C(F)} \chi(c) \leq c q^{-1/2} \]
 where the constant $c$ is bounded in algebraic families, i.e. depends only on $n$ and the degrees of $f$ and polynomials cutting out $C$.  This inequality is known as the   {\em Weil bound}.  
 The ideas used by Weil in this connection, greatly amplified,  became  central in the Grothendeick-Deligne theory, and in equidistribution results; 
 see e.g. \cite{kowalski-icms}.

 None of this, of course, can be represented  
within Ax's theory of pseudo-finite fields, since the additive character $\chi$ is not part of the structure.  
Nevertheless,  Emmanuel Kowalski was able to demonstrate  that   exponential sums and definability in finite fields fit very well
together.  In \cite{kowalski},   he showed how to estimate exponential sums, not only over curves  or varieties but also when $X$ above is a general definable set in the language of rings.

The contribution of the present paper is to allow the additive character to be part of the model-theoretic structure.  
The language thus includes the symbols of the ring language  $=,+,\cdot,0,1$, along with 
  an additional  symbol $\Psi$ whose interpretation is an additive character. 
This changes the meaning of definability:     any function obtained from $\chi$, or from characteristic functions of varieties,
by   $\min,\sup$ and polynomial operators  taken in any order is now viewed as definable.    The measure arising from averaging in finite fields remains   definable, and thus one may include the corresponding integration operators as well.       A quantifier-elimination
 theorem is proved, analogous to Ax's in that the basic functions involve a push-forward from finite covers.  
 As stated in the abstract, this implies a posteriori that any of the functions obtained by sup, min or integration   
 can be uniformly approximated by basic functions.  Model-theoretic properties  become meaningful; we show that the theory is simple, and indicate some interesting phenomena, in particular a nontrivial connected component of the Kim-Pillay group, when it is embedded within the model companion ACFA of the theory of difference fields. 
  
  Adding $\Psi$ to the language  presupposes a passage to continuous logic, able to accommodate functions with real or complex values.     As equality on the field is treated in the usual way (with no structural metric present)  only a very basic form of continuous logic is needed;    we review it below, following the statement of the main theorem.  
    In a structure $F$ for this logic, a formula $\phi(x_1,\ldots,x_n)$ 
 is interpreted as a function $\phi^F:  F^n \to \Cc$, with bounded image (the bound depends only on $\phi$).  In case the image is
 discrete, say  $\{0,1\}$, the 
 pullback of $1$ is called a {\em  definable subset} of $F^n$, or for emphasis a 
 {\em discretely definable subset}; but these can be rare.

For any class of (possibly enriched) fields  $\mathcal{C}$, we define
the {\em characteristic zero asymptotic theory}  of $\mathcal{C}$ to be
  $Th(\mathcal{C}) \union \{2 \neq 0,3 \neq 0 ,\cdots\}$.   As we only consider characteristic $0$ limits in this paper, 
   we will also simply refer to this as the {\em limit theory} of $\mathcal{C}$.
   
 \begin{thm} \label{summary} Consider the characteristic zero asymptotic theory of the class of   finite fields $\Ff_q$ enriched with a nontrivial additive character $\Psi$.   
    \begin{enumerate}
 \item The theory is decidable, and  has an explicit axiomatization $\Pfp$. 
 \item  $\Pfp$    admits quantifier-elimination relative to    algebraically bounded quantifiers.
   \item  If $F$ is an ultraproduct of the enriched fields $(\Ff_q,\Psi)$, the  pseudo-finite measure is definable.   The image under $\Psi$ of a definable subset of $F^n$ with its pseudo-finite measure,   is a finite union of
rational affine subspaces of $T^n$, with rationally weighted Lebesgue measure.  
  \item  $\Pfp$ is   a simple theory:  indeed a higher amalgamation principle holds.
\item  \label{conservative} $\PFp$ is conservative over the pure  language of rings inasfar as definable sets go:   
any definable subset of $F^n$ is already 
  definable in the language of fields.   
  \item  \label{completions-t} The completions of $\Pfp$  are
     determined by their `absolute numbers', i.e. by  the field of algebraic numbers lying in a model 
  of $T$, along with the additive character $\Psi$ restricted to that field.

 \end{enumerate}
     
 \end{thm}  

 These assertions will be explained and proved in \secref{main} and \secref{sec4}.    At this point we will just make a few comments.

(1)   Decidability holds in a strong sense: given a sentence $S$ 
(formed using the basic relations, connectives and quantifiers), and $\e>0$, one can effectively find a
sentence $S'$ such that $|S-S'| < \e$ in any model of $T$, and a number field $L$,  such that $S'$ has only quantifiers ranging $L$; and the set of possible values of $S'$ is $P(\T^n)$ for some explicit polynomial function
$P$ and some $n$.

    Much of the model  theory in \thmref{summary}, including a quantifier-elimination 
 to the level of bounded quantifiers, should   carry thorough for a character of any
commutative algebraic group.  For instance for multiplicative characters, the axioms will concern curves in Cartesian powers of the multiplicative group $G_m$,
that do not identically satisfy  a multiplicative relation $\Pi_i x_i^{m_i}=c$ of bounded height $\sum_i |m_i|$.   
Certainly, different issues will arise in the interface with geometry.    
 For example, some special multiplicative characters -
those whose image   is contained in the $m$'th roots of unity, for fixed $m$ - are already definable in the pseudo-finite field structure, and can be understood within the existing theory; the new limit theory will be valid for large finite fields with sufficiently general (rather than   arbitrary nontrivial) multiplicative characters.     In any case,   for simplicity, we will restrict attention for the present to the additive character.  

 An alternative presentation of the results is possible, employing  a new
sort $\widehat{F}$ for the whole `definable character group', in place of  naming a single specific character.   
   $\widehat{F}$ carries an abelian group structure, as well as an action
of $G_m(F)$, transitive on $\widehat{F} \m (0)$.  In  place of $\Psi$
the language includes a continuous-logic relation $\psi$ on $F \times \widehat{F}$, with values in $\Tt$.  
The universal axioms assert that $\psi$ is $\Zz$-bilinear, and $\psi(cx,y) = \psi(x,cy)$.  In particular,
for any $d \in \widehat{F}$, $\Psi_d(x):= \psi(d,x)$ is an additive character.  
  In addition, $\Pfpw$ asserts that $(F,\Psi_c) \models \Pfp$.      Then $\Pfp$ interprets $\Pfpw$; indeed $\Pfp$ is
  bi-interpretable with $\Pfpw$ if one adds a constant symbol $d$ in the sort $\widehat{F}$ to the language, with the sole axiom
  $d \neq 0$; the interpretation of $\Psi$ will be $\Psi_d$.   Then   \thmref{summary} holds for $\Pfpw$ in full; the completions are now
  determined purely by   their  absolute numbers as fields.  We will not adopt this variation in part since it is special to the additive group, and also so as not to sweep under the rug the natural question discussed in the next paragraph.

On the prime fields $\Ff_p$, one can define a canonical additive character
\[ \Psi_p: \ \Ff_p \to \Tt \]
\[ {a}\bmod{p} \mapsto exp(2 \pi \frac{a}{p}) \]
   The limit theory of the class of finite fields $(\Ff_p,\Psi_p)$ with this specific choice of additive character
contains $\Pfp$, and is generated over it by some quantifier-free sentences.  
All the assertions of \thmref{summary} hold for this theory
with the possible exception of the first, decidability.    An explicit axiomatization $\Ppp$   is suggested, that would imply decidability; but whether or not it suffices to axiomatize the limit theory  hinges on  on a certain open question in number theory,
  close to a well-known problem (and theorem) of Duke, Friedlander, Iwaniec in   \cite{dfi}.  This will be discussed
in \secref{prime}.

The last section contains some   remarks, conjectures and questions regarding extensions to $p$-adic additive characters, and (especially) to difference fields and ACFA.

 For a field $F$, $F^a$ will denote the algebraic closure.
 
 \ssec{Continuous logic}       A {\em continuous model theory} quite able to deal with real- or complex-valued
  functions was already available at the time of Ax's paper; see \cite{chang-keisler-c}.    But it was much more recently,
  following  sustained work of many, that the boundaries with discrete first-order logic were dissolved, extending
   the scope of model theory, almost seamlessly, beyond pure algebra; we refer to  \cite{bybhu}.   Much of the
 effort in developing the general theory arises from a metric, that replaces equality in the discrete case.  For
 our present purposes this metric is unnecessary, and even impossible, i.e. the only metric one could consistently impose is the discrete one.   Thus the framework is a very small variation on the familar case; we briefly present it here.  
  
 Function symbols and the formation of terms  are treated in the usual way.
Basic relations $R$ come with a compact subset $V_R \subset \Cc$, their intended range of values.  (We treat $\Cc$ as having a distinguished $\sqrt{-1}$, so that a complex-valued relation is equivalent, if desired, to two real-valued ones.)
  It is best to think of $\phi$  not as a yes-no question with a smeared out set of possible answers, but simply as a question that has a range of possible answers in the first place.   
In case $V_R = \{0,1\}$, we can call $V_R$ two-valued; in particular this is the case for the equality symbol.   
Formulas are formed using connectives and quantifiers, then closed under uniform limits (uniform over all evaluations in all structures for the language.)   The connectives are complex conjugation and complex polynomial operations  $+,-,\cdot$; however by Stone-Weierstrass we can also throw in any  continuous function from $\Cc^n$ to $\Cc$, without changing the set of formulas.  
  Quantifiers are replaced by continuous $\Cc$-valued maps on the space of compact subsets of a given compact of $\Cc$; suprema and infima of real-valued functions suffice.    
  A structure $\uA$ is a set $A$ along with a function  $R^{\uA} :A^n \to V_R$, for each basic $n$-ary relation $\phi$.  The interpretation $\phi^{\uA}$ of an formula $\phi$, along with the compact set $V_\phi$ in which it takes values,  is then defined in the obvious way.

All the usual definitions and notions of basic model theory generalize readily, once one gets used to the transposition.  For instance, a sentence is a formula with no variables.  For each structure $\uA$,
$Th(\uA)$ is the map $\phi \mapsto \phi^{\uA}$ from the set of sentencs $\phi$ to points in $V_\phi$. 
The theory of a class $\mathcal{C}$ of structures is the assignment of a closed subset to each $V_\phi$,
namely the closure of $\{ \phi^{\uA}: \uA \in \mathcal{C}\}$.  
 A theory is {\em complete} if it gives each sentence $\psi$ a definite value
(in $V_\psi$.)    A complete theory $T$ is  {\em decidable} if  there exists an algorithm that  given 
a sentence $\psi$, and any $\e>0$, is guaranteed to terminate and output
 the value of $\psi$ up to a possible error below $\e$.    Likewise, a possibly incomplete theory is decidable if
 given  a sentence $\psi$, and any $\e>0$, an algorithm can output a finite subset $F$ of $V_\psi$, such that for any
 $\uA \models T$, 
 $\psi^{\uA}$ is at distance at moset $\e$ from  $F$; and conversely any point of $F$ is at distance at moset $\e$ from 
 some such $\psi^{\uA}$.  When $T$ is 
complete and has an explicit (or recursive)
 axiomatization,   it is automatically decidable.
  
   The compactness theorem holds:  if $\Phi$ is a set of pairs  is $(\phi,B)$   consisting of a sentence $\phi$ and 
 a closed subset $B$ of  $V_\phi$,  we say $\Phi$ is {\em satisfiable} if there exists a structure $\uA$ with
 $\phi^{\uA} \in B$ for each pair $(\phi,B) \in \Phi$.  $\Phi$ is {\em finitely satisfiable} if each finite subset is
 satisfiable.  Then any finitely satisfiable set is satisfiable.  
 
 There is also an ultraproduct construction of $\uA$.  
     The theory of 
saturated models generalizes, etc.  In fact deeper theories including stability and simplicity generalize too,
\cite{cats}, but we will not use them here.
 
 Let us recall also   the definition of  quantifier-elimination in this setting:
 
\begin{defn}  $T$ admits quantifier-elimination 
if for any formula $\psi$ and any $\e>0$ there exist atomic formulas $\phi_1,\ldots,\phi_k  $ and a continuous 
function $C$ such that whenever $M \models T$ and $a \in M^x$, we have 
$|\psi(a) - C(\phi_1(a),\ldots,\phi_k(a))| < \epsilon$.
\end{defn}

The usual   criteria for QE go through from the discrete 1st-order logic case.
When $T$ admits quantifier-elimination, a type is determined by a quantifier-free type (and only then.)  
(Proof:  the continuous map restricting complete types to qf types will under these circumstances be a bijection;
as the two spaces are compact Hausdorff, it is a homeomorphism.)  It follows likewise that   $T$ admits QE provided,
for any two   $\lambda^+$-saturated models, $\lambda \geq |L|$,
 any isomorphism between substructures of cardinality $\leq \lambda$ extends to an another
 whose domain includes a prescribed element.

\ssec{Acknowledgements}
 
 Many thanks to Chieu-Minh Tran and to Jamshid Derakhshan for early discussions,
 and to Zo\'e Chatzidakis, Alexis Chevalier, 
 Anand Pillay,  Micha\l{} Szachniewicz,  and  two anonymous referees, for their   comments.

Exponential sums in a model theoretic setting were discussed by Tomasic in \cite{tomasic}, who noted that in positive characteristic the additive character is definable, and used equidistribution to determine the theory of certain 
reducts of pseudo-finite fields.    
Indeed if $(F,+)$ has exponent $p$, a  character $\chi: F \to \Tt$    takes only $p$ values,  the $p$'th roots of $1$;
the kernel of $\chi$ can be taken to be the image of the Artin-Schreier map $\wp(x)=x^p-x$;  and so Ax's
 discrete first-order theory already interprets  the $p$-element group $Hom(F/ \wp(F), \Cc^*)$.    For this reason,
we will consider only the characteristic zero limit in this paper.

Chieu Minh Tran \cite{chieuminh}   used the exponential sum estimates   to determine the (discrete) first order theory of  $\Ff_p^{a}$ with     'multiplicative intervals'.

 Zilber \cite{zilber} used them with a view to quantum mechanical integrals, taking different limits than we do here.
    
  \end{section}

  \begin{section}{A wrong turn}
  
  \label{wrongturn}
  This section is merely a ``no through way" sign on a sideroad; the reader may   pass it by without loss, except perhaps of some  contrast:  what seems like a slight variation  leads inexorably to undecidability.

  Let us identify the universe of $\Ff_p$ with $0,1,\ldots,p-1$, with addition and multiplication modulo $p$.  
Define  $\Lam_p: \Ff_p \to [0,1]$ by $\Lam_p(m)=m/p$; it is   the inverse of the function $[0,1] \meet \frac{\Zz}{p} \to \Ff_p$,
$x \mapsto px \mod p$.   Taking an ultraproduct of these structures, we obtain a pseudo-finite field $F$ with a function 
$\Lam: F \to [0,1]$.  We have $\Psi= e \circ \Lam$, where $e(x)  = e^{2 \pi i x}$.  From the point of view of continuous logic,
this is an interpretation ($e$ being continuous) but not   a bi-interpretation, since $e$ is not quite injective.
We will soon see that $(F,+,\cdot,\Psi)$ does not interpret $(F,+,\cdot,\Lam)$ in any way, though it 
does of course interpret $(F,+,\cdot,\sin \circ \Lam)$, and is in fact bi-interpretable with it.

On first seeing Kowalski's results, I was struck by this statement.      (I understand Lou Van den Dries made similar observations and questions.)   
 
\begin{prop}\label{k07b}[Kowalski]  For an affine algebraic variety $V \subset \Aa^n$, defined over $\Ff_p$,
the number of points of $V$ in $[0,p/2]^n$ is approximately $p^{\dim(V)} 2^{-n}$, {\em unless}
$V$ is contained in some linear hyperplane of $\Aa^n$.  \end{prop}

The caveat on linear hyperplanes has an innocent ring at first hearing, and the statement sounds very much like the geometric input 
needed for a quantifier-elimination theorem, in the style of Ax,  for the theory of pseudo-finite fields enriched with $\Lam$.  
However,    this cannot be the case:

\begin{prop}\label{1}  Let $I_p=\lam(\Ff_p)$ be the image in $\Ff_p$ of an interval in $\{0,\ldots,p-1\}$.   Assume $|I_p|$ 
and $|\Ff_p \m I_p|$ are unbounded.     Let 
\[T = \{\phi:  (\exists n)(\forall p>n)(\Ff_p,+,\cdot,I_p) \models \phi \} , \ \ \ \ \ \ \ \ 
 T^\perp = \{\phi:  \neg \phi \in T \} \]
  Then $T,T^\perp$ are   recursively inseparable.
\end{prop}

\prf   Let $<$ be the image of the archimedean ordering on $\{0,\ldots,p-1\}$.  We may assume $a_p:= |I_p| < p/2$.   By  intersecting the interval $I_p$ with translates, one obtains a uniformly definable family of convex definable sets including   all intervals $[0,a]$ with $a \leq a_p$.  So the restriction of $<$ to $[0,a_p]$ is definable.  Note that $\lceil a_p/m \rceil$ (the integer part of $a_p/m$) is definable for each $m$, and that $\lceil \sqrt{a_p} \rceil$ is definable  since $[0,\sqrt{a_p}/2]$ is the largest segment contained in \[[0,a_p/4] \meet \{x: x^2 \in [0,a_p/2] \} \ \ ;\]
similarly for higher roots.
Now consider  polynomials $F,G$ in variables $x=(x_1,\ldots,x_k)$ and of   some given degree $d$, with non-negative integer  coefficients.  Let $D$ be a set with $T \subset D$ and $D \meet T^{\perp} = \emptyset$.
Then $\Nn \models (\exists x) (F(x)=G(x))$ if and only if 
\[  \ulcorner  (\exists x_1,\ldots,x_k \in [0,\sqrt[d+1]{a_p}] ) F=G \urcorner \in D \]
  Thus we can decide existence
 of solutions of integer polynomials if we have access to $D$, so by Matiyasevich's theorem $D$ is not recursive.     \eprf

\begin{rem} \begin{enumerate}
\item  Following G\"odel, using the Chinese remainder theorem, 
 we could define the exponential function $x \mod{p} \mapsto 2^x \mod p$ on an unbounded interval
 contained in $[0,\log_2(p)]$; 
 then we could   appeal to Davis-Putnam-Robinson instead, or (with more quantifiers) even directly to G\"odel's theorem, in place of Matiyasevich. 
 \item  $T$ is in fact $\Sigma^0_2$-complete:  one can reduce to 
  $T$ the question of whether a given r.e. set $E$ is finite.      Say again that $E$ is defined by  
   $(\exists x) (F(x)=G(x))$, with $F,G$ of degree $d \geq 2$.   It follows from \cite{ingham} 
   that for large $n$,  there exists a prime $p$ with $n^{d+1} <p< (n+1)^{d+1}$;
   so $n =  \ulcorner \sqrt[d+1]{p} \urcorner $.  Thus $E$ is finite iff for all but finitely many primes $p$,
   $(\Ff_p,+,\cdot,I_p) \models    \neg (F( \lceil \sqrt[d+1]{p} \rceil)=G( \lceil \sqrt[d+1]{p} \rceil))$; iff this sentence lies in $T$.      (Here Ingham's theorem was used to obtain a reasonable large `tally'.   We could use the more basic Chebysheff's theorem / Bertrand's postulate instead, obtaining arithmetic up to $\log(p)$ in place of  $\sqrt[d+1]{p} $;
   or a trivial bound on prime gaps, reaching approximately $\log \log (p)$;  all sufficient of course for our present purposes.)  
 \end{enumerate}
 \end{rem}

  Let us note two   consequences of  \propref{1}.   The first is a curious, purely negative alternative proof of another result of Kowalski's, that he drew as a corollary of his positive results.

\begin{cor}  \lbl{intervals}
 Intervals cannot be defined uniformly in $\Ff_p$, unless they are bounded or co-bounded. \end{cor}
\prf  Ax showed that the theory of pseudo-finite fields is decidable, while according to \propref{1}
they cannot be uniformly defined in any decidable theory.  \eprf 
    
     We include one further statement in this vein.

\begin{cor}  Let $X_p \subset \Ff_p$  
 be a  definable set in the theory of pseudo-finite fields of characteristic $0$, or any decidable expansion.      Assume $|\Ff_p \m X_p|$ is unbounded with $p$.  
Let $k(p)$ be the minimal archimedean gap in $\Ff_p \m X_p$, i.e. the minimum $k \geq 1$ such that  
for some $a \in \Ff_p \m X_p$ we have $a+1,\ldots,a+k-1 \in X_p$, $a+k \notin X_p$.      Then $k(p)$ is bounded.  
\end{cor}  
\prf  Let $k=k(p)$.  By taking $u=a$ we see that 
\[\{x: (\forall u)(u \notin X, u+1 \in X \Rightarrow u+x \in X ) \}\] is contained in $[1,k]$; 
it must be equal to $[1,k]$ by minimality of $k$.  Thus an interval of length $k(p)$ is definable.
So by \propref{1} $k(p)$ is bounded, or else the interval is cobounded, and then   $F \m X$ is bounded.  \eprf

 Thus one cannot expect to have a good quantifier elimination  
 for  $Th(\{(\Ff_p,+,\cdot,I_p): p \})$,  hence not for $Th(F,+,\cdot,\lam)$.

I found this situation puzzling; it seems to deprive us of a natural logical
 setting for   \propref{k07b}, playing a role analogous to the theory \cite{ax} of pseudo-finite fields for 
 \cite{cdm}.   
 
   A part of the solution is the use of continuous logic; 
   When intervals are defined in continuous logic, via a map $[0,p-1] \to [0,1]$ or more conveniently
 to the unit circle $\T \subset \Cc$, the nature of the logic  blurs the endpoints of intervals, and thus    removes at least the apparent source of undecidability.  The price one pays, of course, is that while the theory accounts well for intervals
 of length $p/m$ for any given $m$, it cannot access any intervals of  length of order $o(p)$.   
 
 But even the continous logic theory of $(F,+,\cdot,\Lam)$ is undecidable.  Recall that a $\bigwedge$-definable is called
 {\em definable} if its complement is also $\bigwedge$-definable.    
  On   $\Lam \inv [1/3,2/3]$, the  relation $\Lam(x) < \Lam(y)$  is $\infty$-  definable:    since $\Lam(x) < \Lam(y)$ iff $\Lam(x-y) \geq 2/3$.
 Thus $\Lam \inv (2/5,3/5)$ is $\bigwedge$-definable (the conjunction of $2/5<\Lam(x)$ and $x<3/5$, both taken on 
 $[1/3,2/3]$; but clearly so is the complement $\Lam \inv [0,2/5] \union \Lam \inv [3/5,1]$; so $\Lam \inv (2/5,3/5)$ is definable, and $\Lam \inv (<)$ is definable on it. 
 
 \begin{cor}  The continuous logic theory of the class of enriched prime fields $(\Ff_p,+,\dot,\Lam)$ is undecidable.
 \end{cor}
  
  In the next section, we begin studying 
 the theory $\PFp$, with $\Lam$ replaced by $e \circ \Lam$.   For   measure-theoretic or distribution statements (see \secref{expectation}) it is indistinguishable from the theory of $(F,+,\cdot,\Lam)$, and in particular
 does provide a framework including   \propref{k07b}, while remaining in the tame world.

  \end{section}

\begin{section}{The theory of pseudo-finite fields with an additive character}\label{main}
We work in continuous logic, as described above.

\ssec{The language}

The language $L_+$  has a sort $F$ for the field, with equality treated in the usual (discrete) way; with the ring operations.  
  And there is one additional unary relation $\Psi:  F \to \T$, where $\T=\{z \in \Cc: |z|=1 \}$ is the complex unit circle.   
   The function $ (x_1,\ldots,x_n) \mapsto (\Psi(x_1),\ldots,\Psi(x_n)):  F^n \to \T^n$ will  be denoted by  
   $\Psi^{(n)}$, or just    $\Psi$  when no confusion can arise. 
\ssec{Defined terms}  \label{definedterms}  

Let $P(u_1,\ldots,u_n; x)$ and $Q(u_1,\ldots,u_n; x)$ be integral polynomials.   In any field $F$, we define $\kappa(b_1,\ldots,b_n)= c$  if there exists at least one
root in $F$ of $P(b_1,\ldots,b_n)$, and for any such root $d$ we have $Q(b_1,\ldots,b_n,d)=c$.   (If no such $c$ exists,
let $\kappa_{P,Q}(b_1,\ldots,b_n)=0$.)   
The functions $\kappa$ are algebraic functions, and at the same time are well-defined functions in the usual sense;
we will view them as basic terms in the language.  

With these terms in the language, Ax's theory PF (in the usual, discretely valued first order setting) admits quantifier-elimination, and any substructure is definably closed.  

We also include some   defined terms using $\Psi$.    Let
\[ \Psi^n_{sym}(c_1,\ldots,c_n) :=  \sum \{ \Psi(x):  x^n + c_1x^{n-1} + \cdots + c_n = 0 \} \]
  (To be precise, if $x^n + c_1x^{n-1} + \cdots + c_n= \Pi_{i=1}^n (x-\alpha_i)$ with $\alpha_i \in F^{a}$,
 we define $\Psi^n_{sym}(c_1,\ldots,c_n)^F = \sum_{i=1}^n \Psi(\alpha_i)$; where $\Psi(x)=0$ for 
 $x \in F^{a} \setminus F$.)  $\Psi^n_{sym}$ is clearly definable; it is our analogue of the algebraically bounded quantifiers required for quantifier-elimination in the case of Ax.

Thus a basic formula has the form $\Psi^n_{sym}(g_1(u),\ldots,g_n(u))$ with $g_i$ a basic
PF-definable function as above.
 
\begin{rem} \label{lrem} We note a few closure properties of these terms.
\begin{enumerate}[label=(\roman*)]
\item  $ \Psi^1_{sym}(-c)= \Psi(c)$.  In particular  $1= \Psi^1_{sym}(0)$.
\item  $\Psi^n_{sym}(-c_1,c_2,\cdots,(-1)^nc_n)$ is the complex conjugate
of $\Psi^n_{sym}(c_1,\ldots,c_n)$.
\item  $\Psi^n_{sym}(c_1,\ldots,c_n)+ \Psi^m_{sym}(d_1,\ldots,d_m) = \Psi^{n+m}_{sym}(e_1,\ldots,e_{n+m})$
where the $e_i$ are the coefficients of $(\sum c_ix^i)(\sum d_j x^j)$.
\item  Let $\alpha(x,u)$ be any PF formula, such that $PF \models (\forall u)(\exists^{\leq n} x) \alpha(x,u)$.
Then $ \sum \{ \Psi(x): \alpha(x,u) \}$ 
can be expressed as a basic formula.  
(Let $t$ be a variable and write 
  $\Pi \{ (t- c):  \alpha(c,u) \} = \sum_{i=0}^n g_i(u) t^{n-i} $;
then the $g_i(u) (i=0,\ldots,n)$ are PF-definable functions, and $\Psi^{\alpha}(u)= \Psi^n_{sym} (g_1(u),\cdots,g_n(u))$.)
\item  More generally, if $\beta(x,u)$ is a function into $\Nn$ with finite image, whose level sets are PF-definable,
then $  \sum_x \{ \beta(x,u) \Psi(x): \alpha(x,u) \}$ can be expressed as a   basic formula.  (Split $\alpha$ into a disjoint union of $\alpha_i$ on which $\beta$ has 
constant value $v_i \in \Nn$, and apply the previous two items.)
\item   The product $\Psi^n_{sym}(c_1,\ldots,c_n)\Psi^m_{sym}(d_1,\ldots,d_m) $ 
can be expressed as a basic formula.  (Take $\beta$ to give the number of ways that an element can be written
as a sum of a root of $\sum c_ix^i$ and a root of $\sum d_j x^j$.)
 \item  Hence the expressions $\frac{1}{m} \Psi^n_{sym} (u)$ form a ring, closed under complex conjugation.
  By Stone-Weierstrass, 
 an arbitrary continuous function composed with basic formulas can be uniformly approximated by a basic formula.  

\end{enumerate}
 \end{rem}

\ssec{Axioms for  $\PFp$}  \label{axioms}
A hyperplane $Y \subset \Aa^n$ is said to have {\em height $\leq m$} if it can be defined 
by a linear equation $\sum A_i X_i = b$ with $A_i \in \Zz$, $|A_i| \leq m$.

\noindent
\begin{enumerate}
 \item $F$ is a field containing $\Qq$, with a unique 
 Galois extension of order $n$ for each $n$;
 \item   $\Psi: (F,+) \to \T$ is a homomorphism;
 \item   Let $h \in \Qq[z_1,z_1 \inv,\ldots,z_n, z_n \inv]$ be a finite Fourier series 
 ( Laurent polynomial) with degrees $\leq m$, 
 real-valued on $\T^n$, with no constant term. 
  For any absolutely irreducible curve $C \subset \Aa^n$, not contained in any  hyperplane of height at most $m$, 
\[ \sup \{h(\Psi^{(n)}(x)): x \in C \} \geq 0 \]

\item[($\infty$)] Of course, if we add  the defined terms $\kappa, \Psi_{sym}$ to the language, their definition must be included among the axioms; in particular for $\kappa$,
\[  (\forall u_1,\ldots, u_n,x)(P(u_1,\ldots,u_n,x)=0 \to Q(u_1,\ldots,u_n,x)=\kappa_{P,Q}(u_1\ldots,u_n))\]
\end{enumerate}

The essential content of the AE axiom group (3) is  that   $\Psi(C)$ is dense in $\T^n$, provided 
  $C$ is not contained in any {\em rational} affine subspace of $F^n$; see \lemref{modelsdense} below.   One could more easily formulate axioms
    for $C$ not contained in {\em any} subspace of $F^n$; but
that would not suffice.

\begin{lem} \label{modelsdense}  \label{2.7} Let $F$ be a  field of characteristic zero, and $\Psi: (F,+) \to \T$ a homomorphism. Then $F \models \PFp (3)$ if and only if the following condition holds:

(3*)  Let $C \subset \Aa^n$ be an absolutely irreducible curve over $F$, not contained in any rational hyperplane
of $\Aa^n$.  Then $\Psi^{(n)} (C(F))$ is dense in $\T^n$.  
 
 \end{lem}
Here a {\em   rational hyperplane} is one defined by an equation $\sum_{i=1}^n m_i X_i = b$, with $m_i \in \Zz$ and $b \in F$.
 
\prf  Assume $\PFp (3)$ holds.  Let $U$ be a nonempty open subset of $\T^n$.   Find a continuous $g \geq 0$, whose support is contained in $U$,  and with $\int g d \lambda =3$, where  $\lambda$ denotes normalized Haar
 measure on $\T^n$.  Using Stone-Weierstrass,   find a   polyomial $h(x)$ in $z_1,\ldots,z_n,\bar{z_1},\ldots,\bar{z_n}$ with rational coefficients,
 of degree say $m$,  with $|| g-h|| < 1$ in the uniform norm; replacing $h$ by the average of $h$ with
 its complex conjugate, we may assume $h$ takes real values; in particular the constant term $h_0 \in \Rr$. 
   So $h_0= \int_{\T^n} h  \geq (\int_{\T^n} g - 1) > 1$.     By the axiom,  there exists $c \in C(F)$ with $|h(\Psi(c))| \geq 1$.  Thus $\Psi(c) \in U$.  As $U$ was arbitrary,
   $\Psi(C(F))$ is dense in $\T^n$.
   
 Conversely, assume (3*).  Let $C,h$ be as in $\PFp$ (3). 
 Let $(a_1,\ldots,a_n)$ be a generic point of $C$ over $F$.   Then $\{a_1,\ldots,a_n\}$ generate a torsion-free, hence free, subgroup of $\Aa/ \Aa(F)$; let $b_1,\ldots,b_l \in <a_1,\ldots,a_n> $ represent a free generating set.   Writing $a=(a_1,\ldots,a_n), b= (b_1,\ldots,b_l)$,
 we have $b=Ba$ for some $l \times n$- integral matrix $B$; and   $a = Mb+a_0$ for some $n \times l$-matrix $M$
 and $a_0 \in F^n$.  The matrix $B$ defines a linear transformation $B: \Aa^n \to \Aa^l$, and 
 $z \mapsto Mz+a_0$ 
 an affine transformation $\Aa^l \to \Aa^n$; their composition is the identity on $C$, showing that  
 $B$ restricts to an isomorphism $C \to C':= B(C)$.  
 
Write  $h=\sum  c_i h_i$ as a finite linear combination of  Laurent monomials.     For each monomial $h_i$ 
we have $h_i \circ \Psi^{(n)} \circ M = \Psi \circ H_i$ where $H_i(x)=\sum k_i x_i$ is a  linear form with $|k_i| \leq m$;
$H_i$ is nonzero since $h$, by assumption, has no constant term.  The linear form $H_i \circ M$ is also nonzero: 
otherwise $H_i$ would be constant on the image of $L$, but this image contains $C$, so $C$ is contained in a translate
of $\ker(H_i)$, contradicting the assumption that $C$ is not contained in any hyperplane of height $\leq m$. 
 It follows that $h \circ \Psi^{(n)} \circ L = h' \circ \Psi^{(m)} $ for an appropriate real-valued
Laurent polynomial; the  monomials of $h'$ are (constant multiples of) $z^{H_i \circ M}$, so $h'$ also has no constant term.    
Integrating $h'$ with respect to the Haar measure
 on   $\T^m$, we have $\int h' =0$ (since each monomial has integral $0$ on $\T^m$) and so $\sup h' \geq 0$.  By 
  (3*) applied to $C'$,   $\Psi^{(m)}(C')$ is dense in $\T^m$, so $\sup_{C'} h' \circ \Psi^{(m)} \geq 0$.
   This proves (3).      
\eprf

\begin{lem}  \label{axiomshold}  Let $\chi_q$ be any nontrivial additive character on $\Ff_q$, and 
let $F$ be any ultraproduct of $(\Ff_q,\chi_q)$ of characteristic zero.  Then $F \models \PFp$. 
 \end{lem}

\prf  
Let $h$ be a Laurent   polynomial (or finite Fourier series) of degree $\leq m$ in variables $z_1,\ldots,z_n,z_1 \inv,\cdots, z_n \inv$, and taking real values on $T^n$.   By
subtracting the constant term, we may assume it is zero.    

  Let $C \subset \Aa^n$ be an absolutely irreducible curve over $\Ff_p$,   not contained in any  subspace of height at most $m$.  
For a nonzero
$k \in \Zz^n$ of height $\leq m$, we consider $f(x) = k \cdot x = \sum_i k_ix_i$ as a function on $C$.   Then $f$ is not constant on $C$,
so the Weil bound applies.  Thus: 
  \[ |\sum_{x \in C(\Ff_p)} \chi_p(k \cdot x) |\leq b p^{ 1/2} \]
  for some $b$ independent of $p$ and of any parameters needed to define $C \subset \Aa^n$ within an algebraic family.
Now for each  monomial $\Pi_i z_i^{k_i}$ of $h$,  we have
$\Pi_i  \chi_p (x_i)^{k_i} = \chi_p(\sum_i k_ix_i)$.
 It follows that  
\[ |\sum_{x \in C(\Ff_p)} h(\chi_p^{(n)}(x)) |\leq b' p^{ 1/2} \]
for an appropriate $b' >0$ (namely $b$ times the number of monomials in $h$).   Thus 
\[ |C(\Ff_p)| \max_{x \in C(\Ff_p)}  h(\chi_p^{(n)}(x))  \geq  \sum_{x \in C(\Ff_p)} h(\chi_p^{(n)}(x))   \geq -b' p^{ 1/2}   \]
So $ \max_{x \in C(\Ff_p)}  h(\chi_p^{(n)}(x)) \geq -b' p^{ 1/2} / |C(\Ff_p)|$.  
Letting $p \to \infty$, taking into account that $|C(\Ff_p)| \sim p$, we obtain the result.  The proof for $\Ff_q$ is identical. \eprf

\ssec{Quantifier elimination}

   \begin{lem} \label{qe0} Let $F,F' \models \PFp$, let $A,A'$ be substructures of $F,F'$ respectively, for the 
   language enriched with function symbols $\kappa, \Psi^{(n)}_{sym}$ as in \ref{definedterms}.   Let  $\alpha: A \to A'$ be
an isomorphism.  Then $\alpha$ extends to an isomorphism 
$acl(A) \to acl(A')$; here $acl(A)$ denotes the relative algebraic closure of $A$ in $F$. \end{lem}

\prf  Since $A$,$A'$ are closed under the $\kappa$-functions, they are definably closed as substructures of the pseudo-finite fields $F,F'$. 
     We can extend $\alpha $  further to a {\em field}
isomorphism $acl(A) \to acl(A')$.  Thus we may assume $acl(A)=acl(A')=:C$ as fields, $A=dcl(A) \leq C$;
   $F,F'$ have two additive characters  $\Psi,\Psi'$, agreeing on $A$.  We also know that $\Psi^n_{sym} (c_1,\ldots,c_n)$
 has the same value, for $c_1,\ldots,c_n \in A$, whether computed with respect to $\Psi$ or to $\Psi'$.  We have to find in this situation a field automorphism
taking $\Psi$ to $\Psi'$.   
By compactness of $Aut(C/A)$, it suffices to find,  given a finite set $\beta$ of elements of $C$,
  an automorphism $\si$ of $C$ over $A$ such that
$\Psi'( b)  = \Psi(\si(b))$   for $b \in \beta$.     We may enlarge $\beta$ so as to be   $Aut(C/A)$-invariant.    Let $B=A(\beta)$, and $G=Aut(B/A)$.  
Introduce variables $T_b$ for $b \in \beta$.
  and obtain a $G$-action on $\Cc[T_b: b \in \beta]$
(fixing $\Cc$; $(g,T_b) \mapsto T_{  g  (b) }$).   
Consider the linear polynomials $f=\sum_{b \in \beta} \Psi (b) T_b$, $f'= \sum_{b \in \beta} \Psi'(b)  T_b$,
where $T_i$ are variables.   Let $\phi= \Pi _{g \in G} g \cdot f$, $\phi'= \Pi _{g \in G} g \cdot f'$.
These are products of linear polynomials; if we show $\phi=\phi'$, it will follow that $f$ divides $\phi'$, 
so $f = g \cdot f'$ for some $g \in G$; then $g \inv$ would be precisely the $\si$ we are after.

In order to show that $\phi=\phi'$, we compute $\phi $  explicitly  in terms of terms $\Psi_{sym}$
over $A$.  Write $\phi = \sum \phi_\nu \T^\nu$.  Then $\phi_\nu$ is a finite  linear combination, with combinatorial coefficients, of products 
$\Psi(b_{1}) \cdots \Psi(b_n)$, $n=|G|$, with $b_1,\ldots,b_n \in \beta$; we rewrite such products as 
$\Psi(b_1+ \cdots+ b_n)$.   Let $n \beta $ be the set of all $n$-fold sums of elements of $\b$;
$G$ acts naturally on $n \beta$.  
If $b=(b_1+ \cdots+ b_n) \in n \beta$ and 
 $\Psi(b_1+ \cdots+ b_n)$ occurs in $\phi_\nu$, then so will 
$\Psi(hb_1+\cdots + hb_n)$, for $h \in G$; we can thus express $\phi_\nu$ as an
integral   linear combination of certain sums $\sum_{h \in G} \Psi(hb)$, $b \in \beta$.
Write $ \Pi_{h \in G} (t-hb) = t^n +   d_1(b) t^{n-1} + \cdots + d_n(b)$.   Then each $d_i(b) \in A$, and
$\sum_{h \in G} \Psi(hb) = \Psi^n_{sym} (d_1(b),\cdots,d_n(b))$.  This expresses $\phi_\nu$
as a sum of $\Psi^n_{sym} $ applied to elements of $A$, and the same expression is valid for $\phi'_\nu$.
\eprf

\begin{prop}   \label{qe}
$\PFp$ admits  quantifier elimination, in the language   including the defined terms $\Psi^{(n)},\kappa$ of \ref{definedterms}.   \end{prop}
  \prf  We will use the criterion mentioned at the end of \secref{intro}, constructing an  
  isomorphism between saturated models $M,N$   by an inductive back-and-forth procedure,
  starting with an arbitrary isomorphism between   substructures.   
 (see \cite{chang-keisler}.)
 
At a given stage we have    small   substructures $A$ of   $M$ and $B$  of $N$, and an isomorphism between them, 
preserving the field structures and the $\kappa$ and $\Psi$; we are to extend it to an isomorphism of $M$ and $N$.   By \lemref{qe0}, we may take $A,B$ to be relatively algebraically
  closed, as the isomorphism extends.   It suffices to show how to extend the isomorphism by one step, so that the domain is a   small extension $A'$  of $A$ within $M$, containing a given  element $a$.  We may
  take $A'=acl(A,a)$, the relative algebraic closure of $A \union \{a\}$ in $M$.  Since $A'$ is relative algebraic closed,
  it suffices to find a field embedding $A' \to N$ over $A$, preserving $\Psi$, and whose image $B'$ is also relatively algebraically closed; 
for then the terms $\kappa$ and $\Psi^n_{sym}$ can be computed correctly within $A'$ and $B'$.   We may assume $A=B$
and the given isomorphism is the identity.

Consider finitely generated substuctures, say $A(a_1,\ldots,a_n)$ where $a=(a_1,\ldots,a_n)$  now denotes a generic point of a curve $C$ in affine n-space;
since $a \in M^n$ and $A$ is relatively algebraically closed in $M$, $C$ is absolutely irreducible.
By compactness, it suffices to find a generic $b \in C(N)$ with $\Psi^{(n)}(b) = \Psi^{(n)}(a)$.  Genericity of $b$ means here 
that $b \notin A$; by another use of  compactness it suffices to avoid finitely many points $\a_1,\ldots,\a_k \in C(A)$;  
replacing $C$ by the affine curve $C \setminus \{\a_1,\ldots,\a_k\}$, we see that it suffices to find any $b \in C(N)$
with $\Psi^{(n)}(b) = \Psi^{(n)}(a)$.

We  can replace $(a_1,\ldots,a_n)$ by  $\Qq$-linearly independent elements
$b_1,\ldots,b_m$ generating the same additive group as $(a_1,\ldots,a_n)$    (This uses the fact that if $a_1=\sum m_jb_j$ with $m_j \in \Zz$ then $\Psi(a_1) = \Pi  \Psi(b_j)^{m_j}$ in the circle group.)
Now by \lemref{modelsdense},     there exists an $m$-tuple $(c_1,\ldots,c_m)$ of $C(N)$   with the same values 
$\Psi(b_i)=\Psi(c_i)$.  By saturation of $N$, we can indeed find an embedding of $A'$ into $N$, preserving $\Psi$.  
If we knew that   the image of $A'$ is relatively algebraically closed, it would follow that $\Psi_{sym}$  and $\kappa$
are preserved too.  At this point we have no basis to claim this, however.

Call a substructure $D$ of $M$  {\em full} in $M$ if  the restriction homomorphism $\gal(M) \to \gal(D)$ is bijective 
(where $\gal$ denotes the absolute Galois group).   If  $D$ is full in $N$,  and $D \leq D'$, the composition 
$Gal(N) \to Gal(D') \to Gal(D)$ is injective, hence so is the restriction $\gal(N) \to \gal(D')$.   In case $D'$ is relatively algebraically closed it is also surjective, so $D'$ is full.  Assume $A',B'$ contain full substructures $A_0,B_0$.  Let $B''$ be the 
relative algebraic closure of $B'$ in $N$.  
Then $\hZ \cong \gal(M) \cong \gal(A') \cong \gal(B')$.   Consider the homomorphisms $\gal(N) \to \gal(B') \to \gal(B_0)$.
Since the composition $\gal(N) \to \gal(B_0)$ is an  isomorphism, $\gal(B') \to \gal(B_0) $ is surjective.  But a surjective
self-map of $\hat{Z}$ is an isomorphism.  Thus $\gal(B') \to \gal(B_0) $ is bijective; and since $\gal(N) \to \gal(B_0)$
is surjective, $\gal(N) \to \gal(B')$ must be surjective too.   Thus $B'$ is relatively algebraically closed in $N$.

We have shown so far, combining the smallness of the Galois group with \lemref{modelsdense},  that if $D \leq M$ is small
(compared to the saturation)  and full, and $D'$ is an extension of transcendence degree  in which $D$ is full too,
then $D'$ embeds into $M$ as a full (and relatively algebraically closed) substructure.  Thus we will be done as soon as we extend the original isomorphism $A \to B$ to one between full substructures (that we took to be the identity).  

  By \lemref{full}, there exists a full 
subfield $A_0$ of $M$ of transcendence degree $1$ over $\Qq$, and likewise $B_0$ in $N$. Let $A'$ be the
relative algebraic closure in $M$ of $A_0$, and   $B'$ of $B_0$.
 We note in passing that the  Chebotarev field crossing argument does not immediately work   for fields with additive character, since the $\Psi$-structure on a field amalgam is not uniquely determined.   However the argument does show that 
{\em some} amalgam $A''$ of $A',B'$ over $A=B$  exists, with Galois group $\widehat{Z}$, and $A',B'$ relatively algebraically closed in $A''$.    By the previous paragraph,
since $A'$ is full in $M$, $A''$ embeds into $M$ over $A'$.  Likewise $A''$ embeds into $M$ over $B'$.  Thus while
$A',B'$ need not be isomorphic, we can skip them and obtain an isomorphism $A'' \to B''$ with $B'' \leq N$, so that we have an isomorphism between full substructures as required.   
  \eprf

\begin{rem}\label{modelsdense-r}  \lemref{modelsdense} can   be extended as follows:  Let $M$ be an $\aleph_1$-saturated model of $\PFp$,
$M_0 \leq M$ countable, relatively algebraically closed, and $p$ any type over $M_0$ in the language of pseudo-finite fields.  Assume $p(x_1,\ldots,x_n)$  
implies the $\Qq$-linear independence of the images of $x_1,\ldots,x_n$ in $M/M_0$.  Let $\alpha_i \in T$.  Then there exists 
a realization $a=(a_1,\ldots,a_n) \in M$ with $\Psi(a_i)=\alpha_i$.  

Namely, let $M_1'$ be a model of PF extending $M_0$, and with $a=(a_1,\ldots,a_n)$ realizing $p$.  Extend $\Psi$
to a homomorphism on $M_1'$ with $\Psi(a_i)=\alpha_i$; this gives a structure $M_1$ for the language of $\PFp$.
 Extend $M_1$ iteratively to a structure satisfying $\PFp$ (1,2,3*)   (\lemref{2.7}).   Then  $M_1 \models \PFp$.
 By quantifier elimination over relatively algebraically closed substructures, $M_1$ embeds elementarily into $M$ over $M_0$.
\end{rem}

  \begin{lem}  \label{full}  Let $T$ be a theory of pseudo-finite fields of characteristic $0$.  Then there exists $F \models T$
and a subfield $F_1$ of $F$ of transcendence degree $1$ over $\Qq$,  such that the restriction $Gal(F) \to \gal(F_1)$ is an isomorphism.\end{lem}    

\prf  It is easy to give a direct proof, but let us use Ax's theorem instead.  Let $M \models T$, and $F_0 = \Qq^{a} \meet M$.  Then  $F_0$ is the fixed field of some $\si_0 \in Aut(\Qq^{a})$.
 Let $L=\lim_n \Qq^a((t^{1/n}))$ be the field of Puiseux series over $\Qq^a$; so $L$ is algebraically closed.  Lift $\si_0$
 to a (continuous) automorphism of $L$ fixing $t^{1/n}$ for each $n$.  Let $\si_1$ be a topological generator
 of $Aut(L/\Qq_a(t))$.  Let $\si = \si_0 \si_1$, and let $F_1$ be the fixed field of $\si$.  
 $Gal(F_1)$ is generated by $\si$.  Since $\widehat{Z}$ is free, there is a unique homomorphism
$\hat{Z} \to Gal(F_1)$ with $1 \mapsto \si$, and we denote it by $m \mapsto \si^m$.  
 Since $\si_0,\si_1$ commute, if $m \in \widehat{\Zz}$ and $\si^m=1$ then  $ \si_0^m \si_1^m=(\si_0 \si_1)^m =1$ so
 $\si_1^m$ fixes each $t^{1/n}$; from this it follows that $m=0$.  Thus $\gal(F_1)=\widehat{\Zz}$.  Let $(L',\si') \models ACFA$
 with $(L,\si) \leq (L',\si')$.  Let $F=Fix(\si') \leq L'$.  Then $F$ is pseudo-finite and $\gal(F) \to \gal(F_1)$ is an isomorphism.
According to Ax, $F \models T$.
 \eprf

    \ssec{Completions} \label{completions} As in Ax's case, the completions are determined by the `absolute numbers' as a field with additive character.  In more detail,  let $G=  Aut(\Qq^a/\Qq)$ be the absolute Galois group of $\Qq$.
   It follows from quantifier-elimination that the completions of $\PFp$ are determined by pairs $(\si,\psi)$, where $\si \in G$ and $\psi: Fix(\si) \to  \T$ is a homomorphism, 
   extending $x \mapsto exp(2 \pi i x)$ on $\Zz$, and with $\si(\omega) =\omega^k$, $\psi(k/n)=exp(2 \pi i /n)$
   for $\omega^n=1$.    Two such pairs are equivalent iff  
they are isomorphic as fields with additive characters; and any pair $(\si,\psi)$ can  occur.   (It is easy to find by an inductive limit argument, a model of $\PFp$ containing the given structure on $Fix(\si)$.)
      
 \ssec{Universal theory}  $\PFp$ is the model completion of the  theory described in axioms (1,2);
 it is easy to extend any model of (the universal part of) this theory to a model of $\PFp$, using \ref{modelsdense} (3*).   
     
 \ssec{Decidability} $\PFp$ is decidable.  Given a sentence $\psi$ and $\e>0$, one can first look for
 a quantifier-free sentence $\psi'$ and a proof that the values of $\psi,\psi'$ are within $\e/2$.
 Now $\psi'$ concerns a number field $L$, that one can take to be Galois over $\Qq$, with Galois group $G$;
 and is determined by an  element $g \in G$, and a homomorphism
 $\psi:  Fix(g) \to \Cc^*$.  Actually only the values of $\psi$ on finitely many elements $e_1,\ldots,e_k$
 are concerned.   Now $e_1,\ldots,e_k$ generate a finitely generated subgroup of $(L,+)$,
 isomorphic to $\Zz^l$ for some $l$, and we may replace them by a lattice basis for the group they generate; so we  may assume they are $\Qq$-linearly independent.  In this case, $\psi(e_i) \in \T$ can be chosen arbitrarily and independently.  Using this, we can determine the set of possible values of $\psi'$, to $\e/2$-accuracy.   
  
\ssec{Standard exponential characters} 

The {\em standard interpretation} of $\Psi$ for $\Ff_p=\Zz/p \Zz$ is by definition the   character  map 
\[\Psi_p:  \ \  n+ p \Zz \mapsto exp(2\pi i  {\frac{n}{p}})  \]

On  $\Ff_q$ (with $q$ a power of $p$) we set 
\[\Psi_q:= \Psi_p(Tr(a))\]
 where $Tr$ is the trace from $\Ff_q$ to $\Ff_p$.
 
 $\Ff_q^+$ will denote the field $\Ff_q$ with the additive character $\Psi_q$.

   Any other additive character on $\Ff_q$
 has the form $\Psi_q(c x)$, for a unique $c \in \Ff_q$; 
  this is a version of  the statement that  a finite abelian group is isomorphic to its dual.    
Hence the group $\hat{A}$ of all additive characters, along with the evaluation map $\hat{A} \times A \to \T$,
  can be interpreted   in the language naming only one character.

 \begin{prop}\label{standard}   The characteristic $0$ asymptotic theory    
of the standard finite fields-with-additive-character $\Ff_q^+$   is  precisely $\PFp$.    It is equal to  the 
 characteristic $0$ asymptotic theory of all finite fields with a nontrivial additive character; and also
 to the  characteristic $0$ asymptotic theory of all prime fields with a nontrivial additive character.  
 \end{prop}

\begin{proof}  We have seen that $\PFp$ eliminates quantifiers and hence becomes complete upon a description of 
$\Psi$ on $\Qq^a$.   We thus have only to show that for any $\si \in Aut(\Qq^a)$ and any homomorphism 
 $\Psi: Fix(\si)  \to  \T$ vanishing on $\Zz$, 
   there exists an $\Ff_q^+$ approximating $(Fix(\si),\Psi)$.
   
The following explicit statement is a little stronger than what we need:  
   
{\em Let $f(X)$ be an irreducible monic polynomial of degree $d $ over $\Qq$, $K=\Qq[X]/(f)$.
Assume the Galois hull $L$ of $K$ is cyclic of degree $d'$  over $K$, and let $r_1,\ldots,r_b$ be 
elements of $K$ such that   $1,r_1,\ldots,r_b$ are $\Qq$-linearly
independent; also let $r_0$ be the reciprocal of some integer $k>0$.      Let $g_i \in \Qq[X]$ be a polynomial
whose image modulo $(f)$  is $r_i$.   
 Let $U$ be a nonempty open subset of $\T^b$, and let $u \in \T$ satisfy $u^k=1$.    Then there exist
infinitely many primes $p$ and such that $f$ has a root $a$ modulo $p$,  and the splitting field of $f$
over $\Ff_p$ has degree $d'$ over $\Ff_p$; and such that moreover there exists $e$ with 
 $(\Psi_{p^e}( g_1(a)),\ldots,\Psi_{p^e}(g_b(a))) \in U$ and such that $\Psi_{p^e} (1/k) = u$.  }

  Using Cebotarev, we can find an infinite sequence  of primes    $p_m$ and $a_m  \in \Ff_{p_m}$,
 such that      the splitting field of $f$
over $\Ff_{p_m}$ has degree $d'$ over $\Ff_{p_m}$.      We will let $q_m = p_m^{e_m}$ for an appropriate sequence of integers $e_m$. 
Then $\Psi_{q_m}(g_i(a_m)) = \Psi_{p_m} ( e_m g_i(a_m))$, and $\Psi_{q_m}(1/k)=\Psi_{p_m}(e_m/k) $.     Let $c_m =(g_1(a_m),\ldots,g_b(a_m))$.  Then it suffices to show given $\e>0$ that for some infinite set of indices $m$, we can find
$e_m$ such that $\Psi_{p_m}(e_m c_m)$ lies in a prescribed open subset of $\T^b$, and $\Psi_{p_m}(e_m/k)$ is a prescribed $k$'th root of unity.      This is  true by
\lemref{weyl+}, noting that as $m \to \infty$, the  $g_i(a_m)$   avoid any given finite number of $\Zz$-linear 
relations, even modulo $1$
(since the $g_i(a)$ and $1$ are linearly independent over $\Qq$.)  At the same time, we can make sure that $e_m$ is relatively
prime to $d'$; thus  the splitting field of $f$ over $\Ff_{p_m^{e_m}}$ still has degree $d'$ over $\Ff_{p_m^{e_m}}$.

The proof for the class of prime fields with an arbitrary nontrivial additive character is completely parallel; at the end,
having found the sequences $p_m,e_m$,  in place of moving to $\Ff_{p_m^{e_m}}$ we remain with $\Ff_{p_m}$, but
multiply the standard character by $e_m$.
\end{proof}

\begin{lem}\lbl{weyl+}   Let    $c_m \in \Ff_{p_m}^b$ be a sequence of $b$-tuples from prime fields of increasing size.   Assume:     
 for any given nonzero rational vector $(\alpha_1,\ldots,\alpha_b)$, for almost all $m$ it is not the case 
that $\sum _{i=1}^b \alpha_i c_{m,i} = 0$.   Let $U$ be a nonempty open subset of $\T^b$.  Also fix $k \in \Nn$ 
and $l \in (\Zz/k \Zz)^*$.
  Then for arbitrarily large $m$, for some   $e_m \in \Nn$,  $\Psi_{p_m}(e_m c_m) \in U$.
  Morever we can choose $e_m = l \pmod k$.
  \end{lem}
 \begin{proof}   This follows from an effective form of  Weyl's criterion for equidistribution; but it
 also follows from the pseudo-finite version \lemref{weyl+pf}, that we will prove separately.  
 Note that the `moreover' is obvious, since changing $e_m$ by a multiple of $p_m$ does not
 affect $\Psi_{p_m}(e_m c_m)$, and $p_m$ is a unit $\bmod k$.  
   \end{proof}

 \begin{lem}\lbl{weyl+pf}    Let $(F,+,\cdot,\Psi)$ be an  
  ultraproduct of 
 enriched finite fields $\Ff_q^+$.   Let ${n} \in \Nn, c \in F^{n}$ and assume 
  $m \cdot c \neq 0$ for $m \in \Zz^{n} \m (0)$.    Then $\Psi(Fc)=\T^{n}$.  \end{lem}
 
 \begin{proof}  In any case $\Psi( F c)$ is a closed subgroup of $\T^{n}$, using $\aleph_0$-saturation of the ultraproduct; so if it is not all of $\T^{n}$ then for some $m \in \Zz^{n} \m (0)$  we have $  \Psi(F c)^m =1 $,
or $\Psi(F c') = 1$ where $c' = m \cdot c \in F$.  But by assumption, $c' \neq 0$ so $Fc'=F$, and $\Psi(F)=T$. 
\end{proof}

 Note we could not simply apply   Weyl's theorem directly to $\Psi(c)$, since that {\em may} fall into some rational hyperplane.

  \ssec{Simplicity}
  
\begin{prop} \label{amalgamation} $\PFp$ admits   $n$- amalgamation for algebraically independent
 algebraically closed substructures.      By the case $n=3$, $\PFp$ is a simple theory.  \end{prop}
 \prf  In fact a stronger statement is true, namely $n$-amalgamation over $PF$:  let $A$ be any functor on the partially ordered set $P(n)$ of subsets of $[n]=\{1,\ldots,n\}$
 into algebraically closed substructures of a model of $PF$ (see \cite{cigha}).  Let $P(n)^-= \{s \subset n: |s|<n\}$.
 Assume further that for $s \in P(n)^-$, a character $\Psi_s$ is given on $A(s)$, so that $\Psi_s | A(s') = \Psi_{s'}$ when
 $s' \subset s$.  Then there exists $\Psi_{[n]}:A([n]) \to T$ completing the data to a functor into models of the universal part of $\PFp$.
  
  The proof is 
 similar to the proof for pseudo-finite fields \cite{pac}, but moving first to the $\Qq$-linear span of the $n-2$-skeleton,   arguing that the values of $\Psi$ are determined thus far, and that beyond this one has $\Qq$-linear independence
and thus sufficient freedom in the choice of $\Psi$.  More precisely,   
for $|s|=n-1$, 
let $B(s)$ be the $\Qq$-span of $\union_{s'<s} A(s')$.  Let 
 $(e_{s,i}:i \in I_s)$  be a subset of  $A(s)$, mapping to a $\Qq$-basis of $A(s)/B(s)$.  
  We have to show consistency of $\Psi(e_{s,i})=\alpha_{s,i}$, for any choice of 
 $\alpha_{s,i} \in \T$.  
 By an internalization argument as in \cite{pac}, we see that the elements $(e_{s,i}: s \in P(n), |s|=n-1, i \in I_s)$ are 
 $\Qq$-  linearly independent as a set.   (e.g. any linear dependence among the elements $e_{{2,3},i}$ over
 $A(1,2) \union A(1,3)$ would imply, by internalizing $A(1)$ into the base, a similar linear dependence over
 $A(2) \union A(3)$, contradicting the choice of the $E_{{2,3},i}$.)
  Hence by \remref{modelsdense-r}, the PF-type of $A([n])$ along with $\Psi(e_{s,i})=\alpha_{s,i}$ is consistent with $PF+$.
   \eprf

  For a definable group $G$ in a saturated structure and a base structure $A$, $G^0_A$ denotes the intersection of all $A$-definable subgroups
  of finite index; whereas $G^{00}_A$ is the smallest subgroup of $G$ cut out by a partial type over $A$,
  and with {\em bounded} index, i.e. the quotient remains bounded in any elementary extension; in this case the quotient
  admits a natural compact topological group structure.    
    It has long  been open, for discrete first order theories that are  simple, whether $G^{00}_0 $ must equal $G^{0}_0$.  
  $\PFp$    gives a natural {\em continuous-logic} example of a definable group $G$ in a simple theory, where
$G^0_A=G$ for all $A$,  but $ G^{00}_0 \neq G$, and where $G^{00}_A$ does not stabilize with $A$.  Indeed, let 
$G$ be the additive group, and let 
$(a_i: i \in I)$
be a $\Qq$-linearly independent subset of $G(A)$.  Then we have a surjective homomorphism $G \to \Tt^A$, 
$g \mapsto (\Psi(a_i g): i \in I)$; the kernel is clearly $A$-$\bigwedge$-definable.    

\ssec{Galois group}  \label{td}
The compact (Lascar-Kim-Pillay) absolute Galois group  $G_T$ of a theory $T$, whether in discrete or   continuous logic, is defined as follows.  A (hyper)imaginary sort has the form $S=K^n/E$, with $E$ an  $\infty$-definable equivalence relation (without parameters).   If  (for some $K$, or all sufficiently saturated $K$)
$S$ does not grow upon replacing $K$ by an elementary extension, we say that $S$ is {\em bounded}.  One then defines a compact Hausdorff topology, the
 {\em logic topology}, 
on $S$, by taking projections of definable sets (with parameters) to be the basic closed sets.    $G_T$ is by definition the automorphism group of the family of all compact sorts (the permutations that extend to an automorphism of some model.)  It is itself naturally a compact Hausdorff group.  

If $T$ admits $3$-amalgamation over algebraically closed sets in the discrete logic sense, then
$G_T$  is totally disconnected.  For $\PFp$ this is the case by \propref{amalgamation}.    See however \ref{acfa+}.

\ssec{The theory with a space of characters}  Recall the theory $\Pfpw$ from the introduction.  
When a nonzero constant symbol is added to the dual sort $\widehat{F}$, it is clear that $\Pfpw$ becomes bi-interpretable with$\Pfp$.  The results of \thmref{summary} thus generalize easily to $\Pfpw$.   Let us only review the completeness statement,
that becomes somewhat stronger:   the completions of $\Pfpw$ are entirely determined by the field of absolute numbers,
i.e. by the restriction to the field language.  

To see this, let $M,N \models \Pfpw$ with $N$ $\aleph_1$-saturated, and assume $M \meet \Qq^a= N \meet \Qq^a=K$.
   Choose $\chi \in \widehat{F}(M)$
and $\psi \in \widehat{F}(N)$.    Let $(b_i: i \in I)$ be a $\Qq$-basis for $K$.   
Use saturation of $N$ to find $d \in N$ such that $\psi(db_i)=\chi(b_i)$ for $i \in I$.  (For any finite $ \{\b_1,\cdots,\b_k\} \subset I$,
existence of such a $d$ with respect to $\{\b_1,\cdots,\b_k\}$ follows from the axioms, since the image of $\Aa^1$ under
$x \mapsto (x \b_1,\cdots,x \b_k)$ is not contained in any $\Qq$-linear affine space.)  Let $\psi' (x) = \psi(dx)$.
Then by completeness of $\Pfp$ we see that $(N,\psi')$ and $(M,\chi)$ are elementarily equivalent in the language
of $\Pfp$.  It follows that $M,N$
are  elementarily equivalent in the language of $\Pfpw$.

\begin{cor}   The theory of the class of prime fields, enriched by their dual groups and evaluation map, is precisely
$\Pfpw$.  \end{cor}

\ssec{The theory in probability logic}  \label{expectation} An alternative and more conceptual axiomatization appears if one uses probability logic;
specifically the expectation logic in the sense of \cite{approx}.  We will only evoke it in passing.

 We continue to work in continuous logic as sketched in the introduction.  {\em Probability logic}, superimposed
 upon this, consists of additional syntactical operators $E_x$ behaving formally like quantifiers; thus
 if $\phi(x,y_1,\cdots,y_m)$ is interpreted as a function $F^{m+1} \to \Cc$ with bound $N$, then 
  $E_x \phi (x,y_1,\cdots,y_m)$ has free variables among $y_1,\ldots,y_m$, and is interpreted as a function
  $F^m \to \Cc$ with the same bound, intended to denote the exectatation of $\phi$ with respect to $x$.   The 
  probability axioms are:
  
  \begin{enumerate}
\item  $E_x(1) =1$ for the constant function   $1$ (viewed as a function of any set of variables);
\item   $E_x(\phi+\phi')= E_x(\phi)+E_x(\phi')$,  and \\ 
  $E_x(\psi \cdot \phi) = \psi\cdot E_x(\phi)$ when $x$ is not free in $\psi$;
\item  $E_x(|\phi|) \geq 0$.  
 \item$E_x E_{x'} = E_{x'} E_x$ 
\end{enumerate}

Axioms (1-3) ensure that a model carries a definable Keisler measure $\mu$, and $E_x \phi $ is the $\mu$-expectation of $\phi$.    
Axiom (4), asserting the commutativity of the operators $E_x,E_{x'}$,  is the model-theoretic Fubini rule.
We need an additional version of Fubini that we formulate for expansions of fields.  Let $C$ be an absolutely irreducible curve,
and let $f,g$ be two nonconstant regular maps on $C$.     Also let $\phi$ be a $\Cc$-valued definable function on $C$.
For $x \in F$, let $f_* \phi (x) = \sum_{y \in C, f(y)=x} \phi(y)$, the sum being taken over $F$-points of $C$.   
Similarly define $g_* \phi$.       \begin{enumerate} \setcounter{enumi}{4}  
 \item    $E_x f_* \phi =   E_x g_* \phi $  \end{enumerate}
This allows us to define $E_C \phi$ to be the common value of $E_x f_* \phi$, for any nonconstant regular $f$.  This 
 version is suitable for (pseudo-)finite fields, where no Jacobian term is needed.    (This extends   uniquely to an integration theory on varieties, see \cite{oleron} 5.14; but only the case of curves will be needed below.)

Now for the language of rings with an additive character, we can formulate the following axioms:
 
  \begin{enumerate}[label=(\alph*)]
\item   $|\Psi(x)|$ = 1
\item  $\Psi(x+y) = \Psi(x)\Psi(y)$, $\Psi(0)=1$.
\item  For any absolutely irreducible affine curve $C$, $E_C 1 = 1$ and $E_C \Psi = 0$.  
\end{enumerate}

This implies the main axiom scheme of $\Pfp$:   if $C \subset \Aa^n$ is an absolutely irreducible affine curve, not contained
in any $\Qq$-linear hyperplane, let $Y \subset \Tt^n$ be the image of $\Psi^{(n)}$ on a saturated model.  Then $Y$ is 
the support of a measure $\mu$ on $\Tt^n$, such that $\int \psi \theta = E_C \theta \circ \Psi^{(n)}$ for any continuous
function $\theta$ on $\Pfp$.  The $\mu$- integral of  any nontrivial character $z_1^{m_1} \cdots z_k ^{m_k}$ vanishes
on $\Tt^n$, while the $\mu$-integral of $1$ is $1$; so $\mu$ is the Haar measure, and $Y= \Tt^n$.    Then $\Pfp$ implies
 elimination of sup and min quantifiers, while Axiom (c) determines the underlying Keisler measure.
We leave the details to the interested reader.

    \end{section}
   \begin{section}{Definable integration, and comparison to PF}

 \label{sec4}

Let $F$ be an ultraproduct of finite fields   with an additive characer, $F_i$, and let $V$ be an absolutely irreducible variety over $F$.
 Then for almost all $i$, $V(F_i)$ we have the   normalized counting measure $\mu_i$ on $V(F_i)$; the integral 
of $\psi$  under this measure is $\frac{\sum_{v \in V(F_i) } \psi(v) } {|V(F_i)|}$.  
The pseudo-finite 
measure $\mu_V$ is obtained as an ultraproduct of these normalized counting measures.
 
\begin{prop}  The pseudo-finite counting measure $\mu_V$  is definable.  
In other words if $\phi(x,y)$, $y=(y_1,\ldots,y_m)$ is a formula then the map
\[ y \mapsto \int \phi(x,y) d\mu_V(x) \]
is a definable map in the sense of continuous logic.  Moreover when $V$ varies in a definable family of definable sets, $\mu_V$ is uniformly definable.  
\end{prop}

\prf   It will be more convenient to show definability for a different normalization, namely for the ultraproduct
of the counting measure on $V$ divided by  $|F_i| ^{\dim(V)}$.   The dimension of $V$ is definable in definable families, and one can then obtain the probability measure 
upon dividing by $\int 1$. 

The statement  reduces using Fubini to the case of affine curves $C_b \subset \Aa^n$ varying in some definable family.   Using quantifier elimination and \ref{lrem} (vii), it suffices to integrate $\Psi^n_{sym}(c_1,\ldots,c_k)$,
where the $c_i$ are PF-definable functions of the variables.  By the definition of $\Psi^n_{sym}$ as a sum,
this amounts simply to integrating $\Psi(x_1)$ over a  PF- definable subset $e_d$ of another affine curve $C'_{d}$, with $d$ a definable function
of the parameters.  Namely, each $c_i(u)$ is uniformly algebraic over $u$.  Thus the Zariski closure
of $\{(u,c_1(u),\cdots,c_k(u)):  u \in C_b \}$ is a curve $C''_b$.  
Let 
$C'''_b =  \{(x,u,z_1,\ldots,z_k):  (u,z_1,\ldots,z_k)  \in C''_b,  x^n + z_1x^{n-1} + \cdots + z_n = 0 \}$,
and $e_b = \{(x,u,z_1,\ldots,z_k) \in C'''_b:  z_i=c_i(u) \}$ (a PF-definable subset of $C'''_b$.)
Then $\int_{C_b} \Psi^n_{sym}(c_1,\ldots,c_k) = \int_{e_b} \Psi(x)$.  
If necessary, passing from $b$ to a parameter $d$ within $\acl(b)$,  break up $C'''_b$ into irreducible compnents.

 The case where $\psi_d$ is a finite set is directly definable via $\Psi^n_{sym}$,
see \ref{lrem} (iv). 
 Using definability of dimension (specifically of dimension zero),  we can similarly reduce   to integrating $\Psi(x_1)$ over
an absolutely irreducible curve $C_b$.  If $x_1$ is constant on $C_b$, with value $v$, the integral is
$\Psi(v) \mu(1_{\psi_d})$.   Otherwise, by the Weil bound, the answer is $0$.
\eprf

We can define, for any formula $\phi(x)$, the Fourier transform 
\[F(\phi):= \int_{x \in F^n} \Psi(x \cdot y) \phi(x) d\mu(x) \]
\begin{cor}
The Fourier transform of any definable real-valued relation on $\Ff_p^n$ is also definable, uniformly in $p$, and  uniformly in families of definable functions.   
\end{cor}
All definable functions on $F^n$ are bounded, in particular
square-integrable; by the Plancherel theorem the Fourier transform of a  definable function is square-integrable with respect to the unnormalized counting measure, hence has countable support; by quantifier-elimination the support
of the Fourier transform consists of algebraic numbers, if the original function is $0$-definable.     

 In particular, the Fourier inversion formula does not survive the `leading order' continuous-logic ultraproduct that we are using.  It would be interesting to know if it holds using the 'next to leading order' measure of \ref{next}, or if a further degree of precision is needed.   Fourier inversion does hold for almost periodic functions
 (uniform limits of finite linear combinations of functions $\Psi(cx)$); I do not know if it holds for any others or if this is the precise domain.

\begin{prop}\label{images}  The images under $\Psi^{(n)}$ of PF-definable subsets $D$ of $F^n$ (in a sufficiently saturated model)  are finite unions of cosets of
subtori of $\T^n$.  Moreover, for fixed $D$, the pushforward measure $\Psi^{(n)}_* \mu_D$  is a finite, positive $\Qq$- linear combination of Haar measures on such cosets (of dimension equal to $\dim(D)$).  \end{prop}

\prf Consider the topology $\tau$ on $F^n$, whose   closed sets are finite unions of $\Qq$-affine subspaces, cut out 
by  (inhomogeneous) $\Qq$-linear equations.   Let   $\union_{i=1}^m H_i$ be the closure of $D$ in this topology.  
Then $D_i=D \meet H_i$ is $\tau$-dense in $H_i$; and it suffices to prove the lemma for each $D_i$ separately
(for the measure theoretic statement, only the top-dimensional components $D_i$ play a role, the others having pseudofinite measure $0$.)   After an $SL_n(\Zz)$-change of variable, and translation, we may assume $H_i$ is a coordinate subspace, cut out by $x_{k+1}=\cdots=x_n = 0$;
and then we may work in $F^k$ instead.  
By the axioms of $\PFp$, $\Psi$ maps $D_i$ surjectively to $\Tt^k$.  Moreover for any nonzero $m \in \Zz^k$, 
the integral of $\Pi_{i=1}^k \Psi(x_i)^{m_i} = \Psi(\Sum_{i=1}^k m_i x_i)$ under the pseudo-finite measure evaluates to $0$;
hence the pushforward to $\Tt^m$ of the pseudo-finite measure on $D_i$ must be a   multiple $\alpha_i$ of the Haar measure on $\Tt^k$.  Since the total measure of each $D_i$ is rational, we have $\alpha_i \in \Qq^{>0}$.   
\eprf

  \propref{images}    is not directly useful to prove definability, since when $D=D_b$ varies in a definable family, the image tori can jump; notably when $D_b = D \meet H_b$ is a hyperplane section of some fixed definable set $D$,  $\Psi^{(n)}(D \meet H_b)$ has  lower dimension when 
$b$ is rational.   The counting measure is definable {\em despite} this translation.

\begin{cor}  \label{ccfin} Let $D \subset F^n$ be a PF-definable set, and let $h: D \to \Cc$ be a $\PFp$-definable map.  
Then  $h(D)$ has finitely many connected components.  Define $x E_h y$ if  $h(x),h(y)$ lie in the same connected component.  Then $E_h$ is a PF-definable equivalence relation with finitely many classes.\end{cor}
\prf The proof of \propref{images} shows that   
 $\Psi^{(n)}(D) = \union_{i=1}^k C_i$ is a finite union  of cosets $C_i$ of subgroups of $\T^n$; moreover we can write $D= \union_{i=1}^k D_i$ with $D_i = D \meet S_i$ definable with   parameters, in fact $S_i$ a $\Qq$-affine subspace of $F^n$, so that $C_i = \Psi^{(n)}(D_i)$.  Now $h$ can be expressed, after replacing $D$ by a finite cover, as $C \circ \Psi$ with $C: \Cc^n \to \Cc$ continuous.  It follows that each $D_i$ is mapped into a single connected component of $h(D)$, and so the classes of $E_h$ are finite unions of the $D_i$.  
\eprf

\begin{cor}[$\PFp$ is conservative over $PF$ for definable sets.]   Let $(F,+,\cdot,\Psi) \models \PFp$ be sufficiently saturated,   let $X \subset F^n$,
and assume   both $X$ and $F^n \m X$ are $\bigwedge$-definable (with parameters) in $(F,+,\cdot,\Psi) $.    Then   $X$ is definable (with parameters) in $(F,+,\cdot)$.\end{cor}
\prf The characteristic function of $X$ is definable in $(F,+,\cdot,\Psi)$ in these circumstances, 
 and so $E_{1_X}$ is PF-definable by  \lemref{ccfin}. \eprf
 
    \bigskip
    
  In view of this conservation result, one may wonder whether $\PFp$ is at all a proper expansion of $PF$;
  certainly some functions onto the interval $[0,1]$ become definable.  But by no means is this the case for $\Psi$.
   In continuous logic, definability would mean that given any two subsets $A \subset B \subset \T$
 of the unit circle   with 
 $A$ compact and $B$ open,   there exists a formula $\phi(x)$ in the language of rings such that    $\Psi(a) \in A$ implies $\phi(a)$,
 while $\phi(a)$ implies $a \in B$.    This is ruled out by the undecidability results of \secref{wrongturn}.  But let us
 also see it in a more geometric fashion.  Let $F$ be a pseudo-finite field (with nonstandard size $p^*$), and let
 $\e>0$ be standard.   {\em A  definable  set $D$ cannot contain an interval $I$ of length 
 $\geq \epsilon p^*$, 
 unless the complement $D':=F \m D$ is finite.}  To see this recall the {\em stabilizer} of \cite{pac}.  In the special case of 
 rank $1$, it can be defined as follows: let $q$ be some rank-one type of elements of $D'$.  Then
  $St_0(q)$ is the set of elements $a$ with $aq \meet q$ infinite;  
 $St_0(q)$ generates a definable group $S$ in two steps, and $S \m St_0(q) $ is finite;
 if $s_1,\ldots,s_k \in S$ are (algebraically) independent then $\meet (s_i + q) \neq \emptyset$.  
 Now $\Ff_p$ has no proper, nonzero
 definable subgroups, hence neither does $F$.  So $S=F$.  Now  one can easily find  independent 
 $s_1,\ldots,s_k \in S$ such that the union of the translates $s_i + I$ is all of $F$.   But then
 $\meet (s_i + D') = \emptyset$, a contradiction.   
  
 A slightly different argument that PF,$\PFp$ differ:  it is clear that the additive group of an ultraproduct of finite fields with additive character admits a $\bigwedge$-definable subgroup such that the quotient is connected in the logic topology,
 namely the kernel of $\Psi$.  But according to results in \cite{pac} we have $G^0 = G^{00}$ for any definable
 group over any base set, i.e. no nontrivial connected quotients exist.

    \end{section}
    
\begin{section}{Prime fields with  standard character}
  
 \label{prime}

In this section we will consider the class $P+$ of prime fields  with their standard character 
\[n \mod p \mapsto exp(2 \pi i p/n) \]
embodying the natural map $\Ff_p = \Zz/p \Zz  \subset \Rr/ p \Zz \cong \Tt$.  (Note that the last isomorphism is unique
if $\Rr/ p \Zz , \Tt$ are viewed as compact groups, up to inversion.)  

  We have found an explicit theory  admitting quantifier-elimination, and true  asymptotically in the family of finite fields with any nontrivial additive character.   It continues to hold, of course, in the smaller class $P+$.
  So every sentence is equivalent to  a quantifier-free one  (in the language with $\Psi_{sym}$); to axiomatize $P+$
amounts to finding the valid quantifier-free sentences, describing, 
 for each number field $L$,  the possibilities for $F \meet L $ and for $\Psi | (F \meet L)$,
where $(F,\Psi)$ is an ultraproduct of elements of $P+$.  

One  group of additional axioms, the axioms $\SP$ below relating $\Psi$ on $\Qq$
   to roots of unity in $F$, is certainly needed.   
   We conjecture that this gives a full axiomatization of the class, and show that this follows
   from a  number theoretic conjecture,  closely related to work of   Duke, Friedlander, Iwaniec, and Toth.  This is the counterpart  for fields with an additive character of the use made by Ax of the  Chebotarev density theorem.
   
 We start with a description of the axiom scheme $\SP$; it consists of axioms $\SP[1],\SP[2],\cdots$, where   the $n$'th axiom asserts:
   
$\SP[n]$   For some $k<n$ prime to $n$ we have:   \begin{itemize}
\item $\Psi({\frac{1}{n}}) = exp(2 \pi i   k/n)$;  and 
\item $[F \meet \Qq[\sqrt[n]{1}]]  = Fix(\alpha_k)$,
\end{itemize}  $\alpha_k$ being the automorphism   of $\Qq[\sqrt[n]{1}]$ acting on the $n$'th roots of $1$ by $k$'th power.

    For instance
for $n \leq 4$ we have  $\Psi(1)=1$, $\Psi(\frac{1}{2}) = -1$,  $\Psi(\frac{1}{3}) = e^{2 \pi i / 3} $ iff $F$ contains a primitive cube root
of $1$   and $4 \pi i / 3$ otherwise,  
$\Psi(\frac{1}{4}) =  i $  if $i \in F$ and $\Psi(\frac{1}{4}) =-i $ otherwise.    Note that it follows from any $\SP[n]$ that $\Psi(1)=1$, and hence $\Psi(\Zz) = \{1\}$, i.e. $\Psi$ induces a character of $\Qq/\Zz$.

\begin{rem} \label{silanguage}  
  The prime field  axioms are nicer to formulate in the language of difference fields, restricted to the fixed field;
this coincides with the language of pseuo-finite fields with the addition, for each $n$, of an algebraic imaginary constant naming (coherently) a generator $\si_n$ of the Galois group of the order $n$ extension.   
Now  both the character group and the Galois group have a distinguished generator.  In the standard models, the $\si_n$
is   interpreted as  the Frobenius automorphism $Fr$; in particular on the $n$'th roots of $1$ (with $n$ prime to $p$) the distinguished automorphism determines an integer $k(n)$,  namely the unique integer modulo $n$ such that $Fr(\omega^k) = \omega$ for each  $n$'th root $\omega$ of $1$.   In this language the $n$'th  axiom asserts:
  \ \
  \[\hbox{If }\si_n(\omega)=\omega^k  \ \ \hbox{ then }   \Psi({\frac{-k}{n}}) = exp(2 \pi i  /n), \]
relating the choice of generator of the  orientation module of $n$'th  roots of unity to the choice of generator of the order $n$
Galois group, via the additive character `to resolution $1/n$'.   \end{rem}

   The truth of $\SP$ in the prime fields $\Ff_p^+$  with standard character, asymptotically as $p \to \infty$,  is an agreeable exercise.   For instance   $\Ff_p$
contains all $n$'th roots of $1$  iff $p=1 \pmod n$ iff $-1/n$ is represented by $(p-1)/n$ in $\Ff_p$.
In this case we have  $\psi(-1/n) = exp(2\pi i /n \cdot (1-1/p))$, so the limit with large $p$ is 
indeed $exp(2\pi i /n)$.   More generaly if $p=k \pmod n$ then $\Psi_p(-k/n)$ approaches $exp(2\pi i /n)$ as $p \to \infty$, and the value
of $\Psi(1/n)$ follows.

\begin{conj}  \label{conj0} The characteristic zero 
 asymptotic theory of prime fields with their standard additive character is axiomatized by
$\PPp := \Pfp+ \SP$. \end{conj}
 
 Since we already have quantifier elimination, the conjecture must be fundamentally number-theoretic;
we will now attempt to bring this out.    
Let $\Qq^{ab}$ denote the field generated by all roots of unity, over $\Qq$.   
For any field $F$ containing $\Qq$, call a homomorphism $h: F \to \T$ {\em acceptable} if 
 for each $n \in \Nn, n \geq 1$,  $h(1/n)$ is a primitive $n$'th root of unity.  The condition is only on $h|\Qq$; it implies that $h$ vanishes on $\Zz$, and indeed  that the kernel of $h$ is precisely $\Zz$.      If $h$ is acceptable, there exists a unique automorphism $\theta_h$ of $\Qq^{ab}$
such that for any $n$ and any primitive $n$'th root $\om_n$ of $1$, for some $k$ prime to $n$, $\theta(\om_n) = \om_n^k$
where $h(-k/n) = exp (2 \pi i /n)$.    When $F=Fix(\si)$, where $\si \in Aut(L)$, $L$ a finite Galois extension of $\Qq$, 
we say $h,\si$ are {\em compatible} if $h$ is acceptable, and there exists an automorphism of $ L \Qq^{ab}$ extending
both $\si$ and $\theta_h$.   

In case $L \meet \Qq^{ab} = \Qq[\om]$ for some primitive $m$'th root of unity $\om$,
compatibility amounts to the equality $\theta_h(\om)=\si(\om)$.

\begin{prop}\label{conj1}  The following statements are equivalent.
\begin{itemize}
\item[$\bullet:$] \conjref{conj0}.

\item[$\spade:$] \ \ Let $L$ be a finite Galois extension of $\Qq$. . 
Let $\si \in Aut(L)$, and let
  $h: Fix(\si) \to \T$ be an acceptable  homomorphism, compatible with $\si$.  
 Then  $(Fix(\si),+,\cdot,h)$ embeds into 
 an ultraproduct of   enriched prime fields from $P+$.

\item[$\club:$]   
    Let $L$ be a finite Galois extension of $\Qq$.   Assume $L \meet \Qq^{ab} = L \meet \Qq[\om]$,
where $\om$ is a primitive $m$'th root of $1$.  Let   $\si \in Aut(L)$, $\si(\om)=\om^k$,  and let $f$ be an irreducible polynomial over $\Qq$ of degree $d$, with a root in $a \in L$ satisfying $\si(a)=a$.        Then for 
any  $\e>0$ and any $u_1,\ldots,u_{d-1} \in \T$,  there exist infinitely many primes $p \in \Nn$ with $p=k \mod m$, 
  and $b \in \Ff_p$, $f(b)=0$, with  $|\Psi_p(b^i) - u_i|< \e$, $i=1,\cdots,d-1$.  

\end{itemize}\end{prop}

\smallskip
Note that the statement $p=k \mod m$ in the conclusion of  $\club$    could equally well be replaced by 
$|\Psi_p(-k/m) - u_0| < \e$, where   $u_0 = exp(2 \pi i/m)$.

\prf   

We saw above that $\Ppp$ holds asymptotically in the class $P+$.  
 
Let us show $\club$ follows from \conjref{conj0}.  Let $L, \om,\si,f,a,k$ be as in $\club$.     Let $h: Fix(\si)\ \to \T$
be a  homomorphism vanishing on $\Zz$, with $h(a^i) = u_i$, and  
 $h(-k/m)  = exp(2 \pi i /m)$.    
Define an automorphism $\tau$  of $\Qq^{ab}$ by setting   $\tau(\mu) = \mu^{j}$, where $\mu$ is a primitive $n$'th root of unity,
and $h(-j/n) = exp(2 \pi i /n)$;  
note $\tau$  agrees with $\si$ on $\Qq[\om]$.  Since $L,\Qq^{ab}$ are linearly disjoint
over $\Qq[\om]$, one can find an automorphism  $\bar{\si}$ of $\Qq^{alg}$ extending both
 $\si $ and $\tau$.  Then $(Fix(\bar{\si}),h) \models \SP$. 
   Let $M$ be a model of $\Pfp$ extending $(Fix(\bar{\si}),+,\cdot,h)$.  By \conjref{conj0}, $M$ embeds elementarily into an ultraproduct of standard enriched prime fields $(\Ff_p,+,\cdot,\Psi_p)$.  Since in these fields
 $\Psi_p(-k/m)$ approaches $exp(2 \pi i /m)$, it follows that $p=k \mod m$ for almost all of these $p$, proving $\club$.

Now assume $\spade$.   We  will show that any model $M$ of $\PPp$ is elementarily equivalent to 
 an ultraproduct $K$ of  prime fields with their standard character.   By \propref{qe} 
it will suffice if $(K \meet \Qq^a,+,\cdot,\Psi) \cong (M \meet \Qq^a,+,\cdot,\Psi)$.   Thus we have to prove that for any $\si \in Aut(\Qq^a)$  and any compatible homomorphism $h: Fix(\si) \to T$, 
the structure $(Fix(\si),+,\cdot, h)$ is isomorphic to $K \meet \Qq^a$ for some ultraproduct $K=(K,+,\cdot,\Psi)$ of prime fields with standard character.    By compactness (of the class of ultraproducts of $P+$), it suffices to find such an isomorphism on a prescribed number field $L$, Galois over $\Qq$; or even a finitely generated subring $R=\Oo_L[1/N]$ of $L$;  i.e.  a field isomorphism
$\alpha: L \meet Fix(\si) \to L \meet K$, with $  \Psi \circ \alpha =  h$ holding on $R$.  
Enlarging $L$  if necessary, we may  assume it contains a primitive $N$'th root $\om$ of $1$.

We have reduced to a statement that is  nearly provided by 
$\spade$; except that   $\spade$ only speaks of an embedding, whereas we require the image of $\alpha$ to be all of $K \meet L$.  To bridge this gap we employ the
 Chebotarev field crossing argument; see  \cite{FJ} for a full account, including the existence of the auxiliary primes $l$
 that appear below.   
   
 Let $d'$ be the order of $\si$ in $Aut(L)$.    
 Let $l$ be a prime such that $d' | (l-1)$ and $\Qq[\om']$ is linearly disjoint from $L$ over $\Qq$, where
$\om'$ is a primitive $l$'th root of $1$.     Let $\tau$ be an automorphism of $\Qq[\om']$
of order $d'$.
Combine $\si$ and $\tau$ to an  automorphism $\bar{\si}$ of $L[\om']$; then $\bar{\si}$ still has order $d'$.
Let $h': \Qq \to \T$ be compatible with $\bar{\si}$  (i.e. with some  automorphism of $\Qq^a$  extending $\bar{\si}$.)  
 Since $\si(\om)=\bar{\si}(\om)$,  and the pairs $(h,\si)$ and $(h',\bar{\si})$ are both compatible, 
 we have  $h' | \Zz[1/N] = h | \Zz[1/N]$.    
Now $\Oo_L[1/N] \meet \Qq = \Zz[1/N]$.  Thus there exists a homomorphism $\Qq+\Oo_L[1/N] \to \T$ 
extending both $h'  \Zz[1/N]$  and $h | \Oo_L[1/N]$; extend it further to  $\bar{h}: Fix(\rho) \to \T$.
The pair $(\bar{h},\bar{\si})$ remains compatible  since this is tested on $\Qq$.

  By $\spade$  applied to $(Fix(\bar{\si}) ,\bar{h})$, there exists an ultraproduct $(K,\Psi)$ of 
  elements of $P+$ and an embedding $\alpha: (Fix(\bar{\si}),\bar{h} ) \to (K,\Psi)$.  Restricted the ring $R$,
  we have $\bar{h} |R = h | R$ and so  $\Psi \circ \alpha  |R = \bar{h}| R = h |R$.  
It remains only to show that $K \meet L$ is no bigger than $Fix(\si)$.
   If $K \meet L > Fix(\si)$, since $K$ contains $Fix(\bar{\si})$, it would follow that
$K \meet \Qq[\om'] > Fix(\tau)$.  However,    as both $K$ and $Fix(\bar{\si})$ satisfy $\SP$,
the fields $\Qq[\om'] \meet L', \Qq[\om'] \meet K$ are determined by $\bar{h}(1/l)=\Psi(1/l)$ and so must be equal.
  This contradiction shows that  $K \meet L= Fix(\si)$ as required.  

To close the circle we prove   $\spade$ from $\club$.  Let $L,\si,h$ be as in $\spade$.  Let $a$ be a primitive generator for $Fix(\si)/\Qq$, i.e. $Fix(\si)=\Qq[a]$, $d=[Fix(\si):\Qq]$.  Then $h$ is determined by $h(a),\cdots,h(a^{d-1})$ and $h | \Qq$.  In turn, $h| \Qq$ is determined by $h(1/N)$ for the various $N$.
Of course, $h(1/N), h(1/N')$ are determined by the single value of $h(1/(NN'))$.  
Using compactness, it suffices to prove that for any $m \in \Nn$, the field $Fix(\si)$  embeds into 
 an ultraproduct of enriched prime fields from $P+$, in such a way that $h(a^i) = \Psi(a^i)$ for $i=1,\ldots,d$, and
 $h(1/m) = \Psi(1/m)$.

We have $L \meet \Qq^{ab} \subset \Qq[\nu]$ for
some root of unity $\nu$; we may assume $m$ divides the order of $\nu$,
so   that $L[\nu]$ contains a primitive $m$'th root $\omega$ of $1$.
Now $L[\nu] \meet \Qq^{ab} = \Qq[\nu]$.   Extend $\si$ to $L[\nu]$, compatibly with $h$.

    Let $k$ be such that $h(-k/m)= exp(2\pi i / m)$.  By $\SP$ (i.e. acceptability), upon possibly replacing $\si$
    by another automorphism of $L$ with the same fixed field, we have 
     $\si(\omega)=\omega ^k$.  
Let $f$ be the minimal  polynomial of $a$ over $\Zz$.  
By $\club$ for $L[\nu],\si,a$,  there exist primes $p_i=k \mod m$ and $b_i \in \Ff_{p_i}$, $f(b_i)=0$, 
such that 
$|\Psi_p(b^i) -  h(a^i)|< 1/k$, $i=1,\cdots,d-1$.    Since $p_i=k \mod m$, we also have 
$\Psi(-k/m)=exp(2\pi i / m)$.  
Let $(K,+,\cdot,\Psi,b)$ be a nonprincipal ultraproduct of the $(\Ff_{p_i},+,\cdot,\Psi_{p_i})$, and embed $Fix(\si)=\Qq[a]$ in 
$K$ via $a \mapsto b$.    Then $h,\Psi$ agree on $-k/m$ and hence on $1/m$, and on each power
$a^i$.   This finishes the proof.
 
\eprf
 
\begin{rem}   In case $\club$ has counterexamples,   the characteristic zero 
 asymptotic theory of prime fields with their standard additive character is axiomatized by
$\PPp$   along with the list of assertions
\[ (\forall x)(f(x)=0 \to  \bigvee_{i=1}^{d-1}  |\Psi(x^i) - u_i| \geq \e  \vee |\Psi(\frac{-k}{ m}) - e^{2 \pi i/m}| \geq \e  ) \]  
 for each such counterexample. 
 \end{rem}

\ssec{Equidistribution}

It is natural to expect not only existence, but also equidistribution of the primes in $\club$.    We will see that a considerably simpler (but still deeply conjectural!) one-variable equidistribution statement implies it.   The required statement, \conjref{dfi+}, appears to be quite close
to  a hypothesis raised in Duke, Friedlander and Iwaniec \cite{dfi}, that we recall below as a conjecture.
 
\begin{conj} [DFI]  
\label{dfi}  
Let   $f \in \Zz[X]$ be an irreducible polynomial of degree $d \geq 2$.  Let $0 < \alpha < \beta \leq 1$.
Let $P$ be the set of primes. 
Consider  
\[S(x) := \{(p,\nu):  p \in P, p \leq x,  \nu \in \Zz,  0\leq \nu < p, \ f(\nu) \equiv 0 \pmod p \} \]
View $S(x)$ as a probability space, with the normalized counting measure.  Then the probability that
$ \alpha \leq \frac{\nu}{p} < \beta $ approaches $\beta-\alpha$ as $|S(x)| \to \infty$.
\end{conj}

The quadratic case was proved in \cite{dfi} and \cite{toth}; the general case was qualified as `far away' in \cite{dfi},
a view that to my outsider knowledge has not required much revision despite a quarter-century of progress.

We will  need what seems to be a slight extension.
Let $L$ be a  Galois extension of $\Qq$, of degree $\delta=[L:\Qq]$, with ring of integers $\Oo_L$.     Let $P'(L,x)$ denote the set of maximal ideals $\fp$ of $\Oo_L$, whose residue field $k(\fp)$ has cardinality at most $x$.  According to the prime ideal theorem,  $P'(L,x)$ has size asymptotic to $x/ \log(x)$.  
There are at most $\delta$ such ideals above a given rational prime $p$.  If  $k(\fp)$  is not a prime field, then 
$ p^2 \leq | k(\fp) | \leq x$;  
so $p \leq \sqrt{x}$; there are thus  at most $\delta \sqrt{x}$ such primes $\fp$, a negligeable number vis a vis $|P'(L,x)|$;
we will disregard them.    We will thus be interested in the set $P(L,x)$ of primes $\fp \in P'(L,x)$ with residue field $\Ff_p$
for some prime $p$.   We will always let $p$ denote the cardinality of the residue field of $\fp$.

Let $a$ be an element of $L$.   For all but finitely many   primes $\fp \in P(L,\infty)$, we have $a \in \Oo_\fp$, and so $\res_\fp (a) \in k(\fp)= \Ff_p$.  Let $\Psi_p$ be the standard additive character $n \mod p \mapsto exp(2 \pi i n/p )$. 
 Let $\mu_{L,x}$ be the normalized counting measure on $P(L,x)$,
and let  $\mu[L,a;x]$ be the pushforward under the map
\[   \fp  \mapsto   \Psi_p(\res_{\fp}(a))   \]

\begin{conj} \label{dfi+} [Extended DFI]  Let $L$ be a Galois extension of $\Qq$, and $a \in L \m \Qq$.  
Then the probability measures $\mu[L,a;x]$ converge weakly   to  the Haar measure on $\T$ as $x \to \infty$.    \end{conj}
 
One can also state the conjecture in the style of \conjref{dfi}:    the fraction of $\fp \in P'(L,x)$ such that 
$ \alpha \leq \res_{\fp}(a)  < \beta $ approaches $\beta-\alpha$ as $x \to \infty$.

Using Weyl's criterion for equidistribution, as in \cite{dfi} (3,4), a third equivalent form is the truth for each $h \in \Zz,
h \neq 0$ of the estimate:  

\begin{equation}\label{dfi+2}    \lim_{x \to \infty}  \int_{\mu[L,a,x]} \Psi_p(\res_{\fp} (a))^h dp  = 0 \end{equation}

Note that \conjref{dfi+}  considers the value of $\Psi_p$ at a single field element $a$, whereas $\club$ requires
the values at $a,a^2,\cdots,a^{d-1}$.  
Despite this we will see in \propref{dfi++} that  \conjref{dfi+}  does imply \conjref{conj0}.

\begin{rem}  The case $L=\Qq(a)$ of \conjref{dfi+} is equivalent to  \conjref{dfi}.  
 \end{rem}
 
\prf  Assume  $L=\Qq(a)$.  Let $f$ be the minimal  polynomial of $a$ over $\Zz$.   Then for some $n$, $n \Zz[a] \subset \Oo_L \subset \frac{1}{n} \Zz[a]$.  Thus, putting asides the finitely many primes dividing $n$, a prime $\fp$ of $\Oo_L$  in $P'(L,x)$
corresponds to a prime $p \leq x$ of $\Zz$ along with a homomorphism $\Zz[a] \to \Ff_p$; this homomorphism is determined by the image of $a$, i.e. by choice in $\Ff_p$ of a root of $f$.  Thus $\mu[x]$ is also induced by the pairs $(p,b)$ with $p \leq x$ and $b \in \Ff_p$,  
$f(b)=0$.   This is what is counted in   \conjref{dfi}.     \eprf

Going a little further,   evidence that removing the condition  $L =\Qq(a)$ is not entirely unreasonable is given by the following remark, treating cyclotomic extensions of $\Qq(a)$.  We state it for quadratic $a$, but in  the same way we can show that if \conjref{dfi} holds true, then so does \conjref{dfi+}  for the case $L=\Qq[a,\nu]$ with $\nu$  a root of  unity.

\begin{rem}\label{quadratic-club}  \conjref{dfi+} is true for cyclotomic $L$ and quadratic $a \in L$ (using \cite{dfi}, \cite{toth}.)  \end{rem}

\prf  We may assume $L =\Qq[\mu]$, $\mu^N=1$; we have $\si \in Aut(L)$, $\si(a)=a$, $\si(\mu)=\mu^k$.  
Let $H = Aut(L/ \Qq[a] )$; so $H$ is an index-two subgroup of $Aut(L) = (\Zz/N \Zz)^*$.   So $k \in H$. We consider the map 
$\fp \mapsto (\Psi_p(a \mod \fp), p \mod N)$ on primes $\fp$ of $\Qq[a]$ with prime residue field $\Ff_p$.  Let $\mu[x]$
be the normalized counting measure on the set of primes $\fp$ of $\Qq[a]$ with norm $p \leq x$.   We have to show
that the image of this map becomes dense as $x \to \infty$; we will even prove equidistribution in $\T \times H$, 
using Weyl's method.      Thus given a character $z^n$ of $\Tt$ and a   character $\delta: H \to \T$,  
 not both trivial, we must show that 
\[\int   \Psi_p (a)^n \delta(p) d{\mu[x]}(\fp)   = o(1)\]
  In case $n=0$, this does not involve $\Psi_p$, and follows from the Chebotarev density theorem.  If $n \neq 0$,
  we have $\Psi_p(a)^n \delta(p) = \Psi_p( na + r) $ for an appropriate rational number $r$, with $\Psi_p(r) = \delta(p)$.
 The results now follows from the quadratic case of \conjref{dfi}, a theorem of \cite{dfi} and \cite{toth}.  
 \eprf

 \begin{rem}  In \propref{conj1}, we obtained \conjref{conj0} from $\club$ qualitatively.   A similar deduction could be made for the equidistribution versions; as is the case with the usual use of the Chebotarev field crossing argument.  The point is that
the behavoiur of $\si$ on $l$'th roots of $1$ for various $l=k \mod m$ is statistically independent among the various $l$, 
and hence also statistically independent from a fixed event that we would like to measure, at least at the limit.   From a stability point of view,  this is the standard fact that a Morley sequence is asymptotically independent from any given element,
specialized to measure algebras;  see   \cite{bybhu}.
This allows relativizing to such an $l$ without losing track of probabilities.\end{rem}
 
 \begin{prop}  \label{dfi++} Assume \conjref{dfi+}.    Let $L$ be a Galois extension of $\Qq$ and $a$  an element of $L$ with
 $d=[\Qq[a]:\Qq]  >1$. 
 
 $\mu^d[L,a;x]$ be the pushforward of $\mu_{L,x}$ under the map
\[   \fp  \mapsto   \Psi_p(\res_{\fp}(a),\res_{\fp}(a^2),\cdots,\res_{\fp}(a^{d-1}) )  \]
Then $\mu^d[L,a;x]$ weakly approaches the Haar measure on $\Tt^{d-1}$ as $x \to \infty$.  

For any $\si \in Aut(L)$ with $\si(a)=a$ and $\si(\om)=\om^k$, $\om$ an $m$'th root of $1$ in $L$,  $\club$ holds.
 \end{prop}
 
 \prf  Again we use the Weyl equidistribution criterion, this time on $\Tt^{d-1}$.   A character of $\Tt^{d-1}$ has the form
 $(z_1,\ldots,z_{d-1}) \mapsto z_1^{n_1} \cdot \ldots \cdot z_{d-1}^{n_{d-1}}$.  Thus we have to prove that for any
 $(n_1,\ldots,n_{d-1}) \in \Zz^{d-1} \m (0)$, as in \eqref{dfi+2}, we have

 \[    \lim_{x \to \infty}  \int_{\mu[L,a,x]} \Pi_{i=1}^{d-1} \Psi_p(\res_{\fp} (a^i))^{n_i}  = 0 \]
but  $\Pi_{i=1}^{d-1} \Psi_p(\res_{\fp} (a^i))^{n_i}  = \Psi_p \res_{\fp} \sum_{i=1}^{d-1} n_i a^i $;  
 as $a$ has degree $d$, we have  $\sum_{i=1}^{d-1} n_i a^i \notin \Qq$; hence the equation 
  follows from \eqref{dfi+2}, applied to the element  $\sum_{i=1}^{d-1} n_i a^i$ in place of $a$.

 If we are also given a primitive $m$'th root of unity $\om \in L$, and $\si \in Aut(L)$ 
 with $\si(\om)=\om^k$ and $\si(a)=a$,   let $H$ be the subgroup of $Aut(\Qq[\om]/\Qq) \cong (\Zz/m\Zz)^*$  generated by $\si$. 
 So the image of $k$ in $(\Zz/m\Zz)^*$ lies in $H$.  Exactly as in \lemref{quadratic-club}, we obtain 
 equidistribution in $\Tt^{d-1} \times H$.  But this clearly   implies  $\club$.
 \eprf

By way of illustration, consider  the case 
where  $f \in \Zz[X]$ is monic, irreducible, and $\Qq[X]/ f$ is Galois.  Then for (all but finitely many) primes $p$ the
number of roots of $f$ in $\Ff_p$ is $0$ or $d$.   On the face of it, \conjref{dfi} is compatible with a picture where for each prime $p$ where $f$ has roots at all, either all $d$ roots are represented by numbers in $\{0,\ldots,(p-1)/2\}$, or all $d$
roots are represented by numbers in $(p+1)/2,\cdots p-1$;
  as long as the number of primes choosing the first option is about equal to those choosing the second.   
  By the Proposition, \conjref{dfi+} (at least) implies  that this is not the case.  Applying $\Psi$ to the finite set of roots gives an element of   $\T^d / Sym(d)$; by varying $p$ among primes $\leq x$  and taking the limit $x \to  \infty$,
  these elements equidistribute on the subset   $\{(z_1,\dots,z_d)  \in \T^d: \Pi_{i=1}^d z_i = 1\}/Sym(d)$.  
   In particular for many (most) $p$ there will be roots in both lower and upper half-intervals.

We note also that 
\cite{foo} proves $\club$ for the case $L=\Qq[a]$, $k=0, m=1$, assuming the Bunyakovsky conjecture.

    \end{section}
    
\begin{section}{Some open-ended remarks}
      
 
\ssec{p-adic additive character}  The map $\Zz[1/p]/ \Zz \to  \T$, $a \mapsto exp((2 \pi i)a)$, induces a homomorphism
$\Psi_p: \Qq_p \to  \T$.  
Construed in discrete first-order logic, the theory of $(\Qq_p,\Psi_p)$ is  undecidable for reasons
similar to \propref{1}; pulling back appropriate arcs in $T$, and rescaling,
one can interpret  long intervals $[1,\cdots,a]$ in $\Zz/ p^n \Zz$.  
Likewise, the asymptotic theory (over all $\Qq_p$) is undecidable in discrete logic.  However it is natural to expect that the continuous-logic presentation will be decidable, and with definable integration with respect to the p-adic measure, both for a single $\Qq_p$ and asymptotically.   It is relevant that integration with an additive character reduces motivically to 
exponential sums over the finite residue field (Cluckers-Loeser, Hrushovski-Kazhdan).  Developing this theory would be interesting.

 \ssec{Transformal geometry}   

  The   theory of fields with an automorphism $\si$ has a model companion ACFA; see e.g. \cite{ch}.   The fixed field of $\si$ is denoted $k$.  
 The theory of pseudo-finite fields is precisely the theory of $k$; in this sense PF is contained
 in ACFA.  
The Frobenius difference field $K_p=(\Ff_p^a,\si_p)$ with $\si_p(x)=x^p$ is not a model of ACFA, but any nonprincipal ultraproduct 
 $K$ of such difference fields is  (\cite{frob}, \cite{varshavsky}).    Over any difference field, there is a notion of   transformal dimension; it can be defined using transformal transcendence degree.   Over 
  $K$  one can also characterize  the 
 transformal dimension of a
definable set  $D=\{x: \phi(x)\}$ as the unique $n$ such that $D_p:=\phi(K_p)$ has dimension 
$n$ for almost all $p$.  Note that $D_p$ is a constructible set, since $\si_p$ is algebraic.  

Any variety over the fixed field $k$ has transformal dimension $0$; thus the world
of varieties over finite fields is represented within the difference varieties of transformal 
dimension $0$.  However, the latter category is considerably bigger.  For one example, ultraproducts of Suzuki groups, defined using a square root of Frobenius, live therein.

When $D$ has transformal dimension $0$, a finer dimension called the total dimension can be defined,
and one has $|D_p| = O(p^m)$ iff $D$ has total dimension $\leq m$.  
One can thus assign a  measure $\mu_0$ to any  definable set of transformal dimension zero, namely the limit
along the ultrafilter of $|D_p| / p^m$.   This measure depends only on the theory of $K$ (along with any parameters used in $\phi$), and not on the specific presentation as an ultraproduct; it 
 has the same definability properties as shown in \cite{cdm} for pseudo-finite fields.  

An example of a difference variety of transformal dimension $1$ is given by the transformal curve 
$D_f: \ \si(y) - y = f(x)$, where $x$ varies over an algebraic curve $C$, and $f$ is a regular 
(algebraic or transformal) function on $C$.  The specializations $D_{f,p}$ are algebraic curves;    their smooth completions, were  used by Weil (\cite{weil}) to bridge the gap between varieties and exponential sums.   In a different way, not requiring smooth completions, the $D_{f,p}$   were used by  Grothendieck and Deligne to  express additive characters  in terms of $l$-adic monodromy.  Katz  raised the question of a uniform treatment 
of such situations; Kazhdan and Kowalski suggested specifically that model theory may be
useful; the model theory of difference geometry seems to be a very natural framework.  

A survey of the model theory of ACFA, with relevant open questions, is planned.  It is difficult to include a summary here of reasonable length, but some minimal  observations seem to be called for.  
 
 \ssec{ACFA with additive character on the fixed field}  \label{acfa+}
 
Consider a characteristic $0$ model $K \models ACFA$, with constant field $k$, and expand $k$ to a model of $\PFp$, in the language considered above.  \footnote{In the same language, we can extend the additive character to $k^{a}$.
Namely, on a finite extension $L$ of $k$, define  $\psi(x) = \Psi (\frac{1}{[L:k]} tr_{L/k} x)$.  We will not use this extension at present.}

  This simplest expansion requires no additional prepatory work:  by the stable embeddedness of $k$ in $K$, or directly from the nature of the quantifier-elimination of $K$, the new theory - let us temporarily name it 
$ACFA_+$ - admits quantifier-elimination.   Nevertheless, it presents already some   aspects worth noting.

(i)   In any continuous-logic theory, and any sort $S$ of that theory, one can define a compact analogue of the absolute
Galois group along the lines of \cite{kim-pillay}.  Namely, consider all $\bigwedge$-definable equivalence relations $E$ on $S$
with a bounded number of classes, independently of the model.   There exists a smallest such $E$, called $E_{KP}$ (recall we do not 
employ the convention of allowing arbitrary parameters.)  If $M$ is a saturated model, the set of $E$-classes carries
a natural compact topology, the logic topology, and the permutations induced from $Aut(M)$ become a compact topological group.   This group coincides with Lascar's $G_c$, the compact part of the general Lascar group \cite{lascar}; in simple theories, it is the full Lascar group.  We will call it the Kim-Pillay group.  
 
Let $\CG$ be the Kim-Pillay  $ACFA_+$.   
By contrast with ACFA (see  \remref{td},)  for $ACFA_+$, this compact   group has a nontrivial connected component
of the identity, $\CG^0$.   
If $B$ is any $0$-definable torsor of the additive group $(k,+)$, one can define an equivalence relation $E$ by:  $xEy$ iff $\Psi(x-y)=0$.  Then $B/E$ is a $T$-torsor.
In case $B$ determines a complete type over $\emptyset$, the automorphism group of $B/E$ is $T$.  For instance, if $a$ is any $0$-definable element of $K$, then 
$B_a=\{y: y-y^\si = a\}$ is such a torsor $T_a$.  In fact $\CG^0$
admits a homomorphism   into $\T^m$ for each algebraic element of $K$ with $m$ conjugates.  Of course, the image of $\CG^0$ in $\Pi_{a \in acl(0)} Aut(T_a) $ is not all of $\T^{\acl(0)}$, but reflects the additive relations among the elements $a$.  

(ii)  Conversely,   I believe it can be shown that $\CG^0$ is commutative, and may be 
precisely the above image.   Let $F$ be a   difference field, relatively algebraically closed in $K$, and let   $F'$  
be the base  structure  consisting of $F$ and $B/E$ as above, for all $\acl(F)$-definable $k$-torsors $B$.    The key is to prove   $3$-amalgamation over $F$; we are given, symbolically, the $2$-types of each pair from $a,b,c$ over $F'$, compatible on the $1$-types and  with each $2$-type independent over $F$, and must find an independent $3$-type extending them.  
As in \ref{amalgamation}, the purely algebraic amalgam is known to exist;
and to define $\Psi$, 
it suffices to show that the known values of $\Psi$ on $k(Fab),k(Fac), k(Fbc)$ where $Fab$ denotes the relative algebraic closure of $F(a,b)$, are compatible with any
 additive relations among them.     
 If $X$ is an  $F^{a}$-definable $k$-torsor,
 and $x \in X(Fa),y \in X(Fb), z \in X(Fc)$, then $x-y,x-z,y-z$ are such a triple of elements whose sum is zero.  
By methods similar to \cite{ad1}, it should be possible to show that all additive relations
on elements of $k(Fab),k(Fac), k(Fbc)$ are generated in this way.   
 These are taken into account in advance in the $1$-types, since  $\Psi$ can be defined on the torsor
$X$, with values in a torsor of $T$.  
 (We may replace $a$ by the $k$-internal part of $a$ over $F'$.  In case $a,b,c$ lie in the same principal homogenous 
space $X$ for a $k$-definable  group $G$, we have $F(k(Fab)) =F(d)$ for $d$ a generic of $G$, and likewise $e$ for $bc$ and
$f$ for $ac$; if any unexpected additive relation holds, one can find $d' \in F(d)$, $e' \in F(e)$, $f' \in F(f)$ with $d'+e'+f'=0$;
then a standard argument shows yields a homomorphism from $G$ to the additive group   $G_a$, say with kernel $N$; and the relation is already accounted for by the images of $a,b,c$ in $X/N$.  In general,  $a$ can be taken to lie in a translation variety in the
 sense of \cite{ad1}, and the situation requires some further analysis.)


(iii)  Let $f$ be a regular function on a curve $C$.  We can also view $C$ as a transformal curve (as such, it is denoted $C[\si]$ in \cite{frob}.)   The transformal curve $D$ defined by $\si(y)-y = f(x)$ is a transformally \'etale  cover of $C$ (smooth and zero-dimensional).  Define $D_\Psi$ as the quotient of
  $D \times \T$ by the identification of $ (d,t) $ with $ (d+a,t\Psi(a))$; we obtain a circle bundle; if instead we
take the quotient of $D \times \Cc$ by the same relation, we find
an archimedean analogue of an $l$-adic local system over $C$.  
The structural automorphism 
$\si $ lifts, given an element $a$ of $C(k)$, to translation by $\Psi(f(a))$ on $D_\Psi$.

$l$-adic local systems are a central object of Grothendieck's theory leading towards the Weil conjectures; part of the 
reason for taking $l$-adic coefficients is their approximability by finite coefficients, where Grothendieck proved constructibility
theorems for cohomology.  Deligne then had to devise methods to move over to archimedean fields, first in his solution 
of the Weil conjectures, but again in the case of his equidistribution theorem.   As explained in \cite{kowalski-aspects}, 
Deligne obtains equidistribution in terms of the natural measure on conjugacy classes in compact Lie groups by what feels like a 
 Robinson-style model-theoretic transfer; the compact groups themselves do not  really make a geometric appearance.  
 
Could one dream of a direct  archimedean analogue of  Grothendieck's   $l$-adic local systems, with his
constructibility   theorems replaced by quantifier-elimination results?  The fundamental group of a variety $V$ is a quotient
of the Shelah-Galois group of $ACF_{F(V)}$:  it is the limit of $Aut(F(V')/F(V))$, where $V'/V$ is finite and unramified.  When one moves to simple theories, such as $ACFA_+$, the Kim-Pillay Galois group has quotients that may be viewed as a generalized fundamental group, no longer totally disconnected.  
In  view  of the probable abelian nature of the connected parts for $ACFA_+$ noted in (ii), and other reasons,
 it seems clear that $ACFA_+$ is not the full answer, and perhaps one needs to go beyond the Kim-Pillay part of Galois.  In any case the possibility is tantalizing.
   
  \ssec{Difference geometry in transformal dimension one}      

   Let us also  take a quick look at how Weil's ideas in \cite{weil} may generalize.  
A notion of smooth transformal varieties exists; it may be defined in elementary terms using the usual Jacobian criterion
applied to difference polynomials in place of polynomials, 
where differentiation  treats any monomial $X^\si$ as a constant.  There is also a notion
of transformal blowing-up, \cite{frob}.  It is plausible that the projective completion of the curves $D_{f,\si}$ can be made
smooth upon blowing up, at least for ordinary polynomials $f$.   (The existence of a smooth difference curve
with a given function field has not been investigated, but carries its own interest.) 
  Assuming   a smooth projective model $E=E_{f,\si}$ exists,
a moving lemma on $E^2$ can be formulated, but has only been proved   for the transformalization of smooth projective algebraic varieties; it plausibly follows from a transposition to difference geometry of the classical `synthetic geometry' proof in \cite{hoyt}.   
   One can now define an intersection product on $E \times E$.
    The coefficient ring is a `motivic' ring constructed out of zero-dimensional difference varieties, with a dynamic `preservation of number' principle built in.     
     Whereas the Grothendieck ring of  zero-dimensional schemes 
has roughly the complexity of the natural numbers (the number of points with multiplicities), zero-dimensional difference schemes have at least  the
complexity of zeta functions; in particular one can apply Frobenius specialization for almost all $p$, to obtain a sequence of ordinary numbers.  Using \cite{frob}, this coefficient ring  admits
 a natural homomorphism into the field $\Rr(\e)$, and so one can work 
 with these coefficients when interested in `next to highest order' estimates (see below).    A Hodge index inequality is valid, though a purely transformal proof is not at present known to me; it follows using the main theorem of \cite{frob} from the case of ordinary algebraic varieties; presumably a direct proof is also possible.   A Weil-style trace can be defined for correspondences on $E$ using  the intersection product with the diagonal; and exponential sums can be related
 to the trace of the structural automorphism $\si$.

%

\bigskip

 \ssec{Next-to-leading-order measure}  \label{next}
 
 Consider  a definable family $(D_a: a \in P)$ of definable sets over a pseudo-finite field $F$.
 We take $F$ with the induced structure as a fixed field of a model of ACFA. 
 (It is very likely that the statements below are true   generally for definable sets of
 transformal dimension zero over $K \models ACFA$.)  
 
 Using Weil's Riemann hypothesis for curves, Will Johnson has shown \cite{johnson} that
 counting modulo a prime $l$ is definable.  Equivalently,  the function $a \mapsto |D_a|$
 from $P$ into $\Zz$, if viewed as a function into the compact set $\Zz_l$, the $l$-adic completion of 
 $\Zz$, is definable in the sense of continuous logic.  \footnote{For $l \neq p$ also follows, though  more elaborately,  from Grothendieck's $l$-adic theory, likewise dependent on Weil's results; see \cite{SGA4.5}     Theorems 4.4.10, and 7.1.1; see also 
 p. 31 for the explanation of the notion of orientation, corresponding to the
richer-than-pure field language we take here.  (See also \S 6.2 of \cite{johnson}.)}

 For an archimedean analogue, it is necessary to renormalize  since the image of $\Zz$ in $\Rr$
 is not relatively compact.  We defined $\mu_0$ by renormalizing by $p^{-\dim}$.  
 The definability of $\mu_0$ implies in particular
 that the relation $\mu_0(X)=\mu_0(X')$ is definable on definable families.  
 
   If we consider the formal expression
 $[X]-[X']$, it can be viewed as a function whose $\mu_0$-integral is $0$.  (If $X,X'$ are subsets of some 
 ambient definable set $D$, we can represent $[X]-[X']$ by a function on $D$, the difference of characteristic functions $1_X - 1_{X'}$. )  
 We can now go one step further, and consider the counting measure $\mu_1$ normalized so that 
$\int (1_X - 1_{X'}) d \mu_1$
is finite and (in general) nonzero; namely the ultralimit of  
\[ q^{\frac{1}{2}-\dim(X)}( |X| -|X'|)\]

  Of course, it can be efficient to   use a single invariant combining $\mu_0$ and $\mu_1$; simply let 
  $\mu(D) = \mu_0(D) + \e \mu_1(D)$.  This takes values in the ordered field $\Rr(\e)$, with $0< \e < 1/n$ for each $n$.  (So $\e^2$ `is' $\frac{1}{p}$.)   But here we will consider them separately.
 
Now $\mu_0$ is automatically definable, with discreteness properties.       $\mu_1$ is not discrete, even in pseudo-finite fields, 
and  cannot be expected to be definable in the pseudo-finite field itself.
  But  in continuous logic, the measure $\mu_1$ can be added to the structure; note that
  $\mu_1$ is bounded in bounded families.    
     One would like to know if $\mu_1$ is definable in a tame geometric expansion of the theory in continuous logic; to begin with, as an expansion of $PF$.  
     Definability of $\mu$ on   families of curves (one-dimensional integration)
 would imply definability over all varieties.    
  
     Here is a    precise question:  let  $(F,\mu_1)$ be an ultraproduct of finite fields with the $p^{-1/2}$-normalized counting measure.  Is it true that 
 $Th(F,\mu_1)$ is simple  as a continuous logic structure, and every definable subset of $F^n$ is     definable over the pseudo-finite field $F$ alone?

 \end{section}

\end{document}